\documentclass[oneside,reqno,11pt]{amsart}
\usepackage{epsfig}
\usepackage{color}
\usepackage{amssymb,amsmath,amsthm,amstext,amsfonts}
\usepackage[normalem]{ulem}
\usepackage{amsmath,amscd,amstext,amsthm,amsfonts,latexsym}
\usepackage[colorlinks=true,linkcolor=blue, urlcolor=black, citecolor=blue]{hyperref}
\hypersetup{
 citecolor=blue}
\theoremstyle{Theorem A}
\theoremstyle{Theorem B}
\theoremstyle{Theorem C}
\theoremstyle{Theorem D}
\theoremstyle{Theorem E}

\makeatletter

\theoremstyle{definition}
\newtheorem{maintheorem}{Theorem}

\numberwithin{equation}{section}
\numberwithin{figure}{section}
\theoremstyle{plain}
\newtheorem*{cor*}{\protect\corollaryname}
\theoremstyle{plain}
\newtheorem{thm}{\protect\theoremname}[section]
\theoremstyle{definition}
\newtheorem{defn}[thm]{\protect\definitionname}
\theoremstyle{question}

\theoremstyle{remark}
\newtheorem{rem}[thm]{\protect\remarkname}
\theoremstyle{plain}
\newtheorem{prop}[thm]{\protect\propositionname}
\theoremstyle{plain}
\newtheorem{lem}[thm]{\protect\lemmaname}
\theoremstyle{plain}
\newtheorem{cor}[thm]{\protect\corollaryname}

\usepackage{xcolor}

\usepackage{geometry}
\usepackage{amsmath}

\numberwithin{equation}{section}
\numberwithin{figure}{section}
\usepackage{enumitem}		
 \let\footnote=\endnote
\@ifundefined{lettrine}{\usepackage{lettrine}}{}

\theoremstyle{definition}

\def\R{\mathbb{R}}

\def\H{\mathbb{H}}
\def\N{\mathbb{N}}
\def\Z{\mathbb{Z}}

\usepackage{float}


\subjclass[2000]{}

\def\Si{\Sigma}

\def\R{\mathbb{R}}

\def\loc{\text{loc}}
\def\H{\mathcal{H}}

\def\vep{\varepsilon}
\def\A{\mathcal{A}}
\def\B{\mathcal{B}}

\def\M{\mathcal{M}_{\text{inv}}}

\def\Sig{\Sigma}
\def\glr{\text{GL}(d,\R)}

\makeatother

  \providecommand{\corollaryname}{Corollary}
  \providecommand{\definitionname}{Definition}
  \providecommand{\lemmaname}{Lemma}
  \providecommand{\propositionname}{Proposition}
  \providecommand{\remarkname}{Remark}
  \providecommand{\theoremname}{Theorem}
\providecommand{\theoremname}{Theorem}

\usepackage{tikz,xcolor,hyperref}

\definecolor{lime}{HTML}{A6CE39}
\DeclareRobustCommand{\orcidicon}{
	\begin{tikzpicture}
	\draw[lime, fill=lime] (0,0) 
	circle [radius=0.16] 
	node[white] {{\fontfamily{qag}\selectfont \tiny ID}};
	\draw[white, fill=white] (-0.0625,0.095) 
	circle [radius=0.007];
	\end{tikzpicture}
	\hspace{-2mm}
}
\foreach \x in {A, ..., Z}{\expandafter\xdef\csname orcid\x\endcsname{\noexpand\href{https://orcid.org/\csname orcidauthor\x\endcsname}
			{\noexpand\orcidicon}}
}

\renewcommand{\le}{\leqslant}
\renewcommand{\leq}{\leqslant}
\renewcommand{\geq}{\geqslant}
\renewcommand{\ge}{\geqslant}

\author{Reza Mohammadpour \orcidA{} and Paulo Varandas
\orcidB{}}

\address{Reza Mohammadpour, Department of Mathematics, Uppsala University, Box 480, SE-75106, Uppsala, SWEDEN.}
\email{reza.mohammadpour@math.uu.se}

\address{Paulo Varandas, Center for Research and Development in Mathematics and Applications (CIDMA), 
Department of Mathematics, University of Aveiro, 3810-193 Aveiro, Portugal \& Departamento de Matem\'atica, Universidade Federal da Bahia, Salvador, Brazil.}
\email{paulo.varandas@ufba.br}

\date{\today}


\begin{document}

\title[Multifractal formalism for Lyapunov exponents]{Multifractal formalism for Lyapunov exponents \\of fiber-bunched linear cocycles }


\maketitle
\begin{abstract}
We develop a higher-dimensional extension of multifractal analysis for typical fiber-bunched linear cocycles. 
Our main result is a relative variational principle, which shows that the topological entropy of Lyapunov exponent level sets can be approximated by the metric entropy of ergodic measures fully concentrated on those level sets, addressing a question posed by Breuillard and Sert \cite{BS21}. We also establish a variational principle for the generalized singular value function. As an application to dynamically defined linear cocycles, we obtain a multifractal formalism for open sets of $C^{1+\alpha}$ repellers and Anosov diffeomorphisms.
\end{abstract}

\section{Introduction}

The multifractal formalism in dynamical systems aims to describe the fine structures of level sets, usually obtained by limit theorems, and has attracted significant interest from the mathematics and physics communities. The most classic example is the level sets defined by the convergence of Birkhoff averages, which we now detail. 
Given a continuous map $T$ on a compact metric space $X$ and a continuous observable $f: X \to \mathbb R$, Birkhoff’s ergodic theorem ensures that sequence
$ 
\big(\frac1n\sum_{j=0}^{n-1} f(T^j(x))\big)_{n\ge 1} 
$ 
is almost everywhere convergent, with respect to any $T$-invariant probability measure.
This motivates us to consider the multifractal decomposition 
\begin{equation}
\label{eq:multifractal-decomposition}
X=\bigcup_{\alpha \in \mathbb R} E_f(\alpha) \; \cup \; I_f
\end{equation}
where 
$$
E_f(\alpha)=\Big\{ x\in X \colon \lim_{n\to\infty} \frac1n\sum_{j=0}^{n-1} f(T^j(x)) =\alpha \Big\}
$$
and the set $I_f$ denotes the irregular set, constituted by points whose Birkhoff averages 
do not converge. The complexity of the level sets $E_f(\alpha)$ can be measured in terms of entropy (see e.g. \cite{Bowen-entropy}) or Hausdorff dimension (see e.g. \cite{BaSc}), just to mention a few. Early contributions on the multifractal analysis can be traced back to Besicovitch and the global picture is nowadays quite well understood especially in the case of hyperbolic and nonuniformly hyperbolic dynamical system and H\"older continuous observables (cf. \cite{BSS, FFW, IJ15,JJOP, JT21,O, Olsen, Shi} and references therein).  Moreover, even though the set $I_f$ is negligible from the ergodic viewpoint, it can be a Baire generic set with full topological entropy or  full Hausdorff dimension (see e.g. \cite{BaSc,BoVa,CaVa}, just to mention a few).

\smallskip 

The multifractal formalism for Birkhoff averages of vector valued maps over the shift was initiated in \cite{BSS,FFW}, where the authors consider a higher-dimensional version of the multifractal decomposition ~\eqref{eq:multifractal-decomposition} by studying level sets $E_f(\vec \alpha)$, $\vec\alpha\in \R^k$ determined by the Birkhoff averages of H\"older continuous observables $f: X \to \mathbb R^k$, $k\ge 1$. 
A first key distinction arises between the one-dimensional and higher-dimensional settings, since when $k\geq 1$, the higher-dimensional spectra may form a non-convex domain and may even have empty interior (see \cite{BSS} for more details).

\smallskip 
 In recent years, motivated by several applications to thermodynamic formalism, dimension theory, and large deviations, the classical notion of potential in dynamical systems has been extended to more general frameworks that include sequences of continuous potentials $\mathcal F=(f_n)_{n\ge 1}$, under mild additivity assumptions. Some contributions on the description of the complexity of level sets 
\begin{equation}
\label{eqdefnonadF}
E_{\mathcal F}(\alpha)=\Big\{ x\in X \colon \lim_{n\to\infty} \frac1n f_n(x) =\alpha \Big\}    
\end{equation}
and irregular set $I_{\mathcal F}$
 include \cite{BoVa2,Cao,CFH,CaVa,falconer94b,O}, with important applications.
Among the more general results, we emphasize that Feng and Huang \cite{FH} computed the topological entropy of level sets for sequences of vector valued asymptotically additive potentials (we refer the reader to \cite[Theorem~5.2]{FH} for the precise statements). 
 Nevertheless, Cuneo \cite{Cuneo} revealed that asymptotically additive sequences of real-valued potentials do not fundamentally extend the scope of the theory beyond the standard additive case for continuous observables.


\smallskip

 Lyapunov exponents can be used to quantify the exponential rates of divergence or convergence of nearby trajectories in smooth dynamical systems. They characterize chaos and stability through the sensitivity of orbits to initial conditions and are used to study the statistical properties of dynamical systems in ergodic theory.
These exponents arise from the exponential growth rates in derivative cocycles, which are our primary focus and belong to the broader class of linear cocycles.
 For that reason, in this paper we address the multifractal analysis of Lyapunov exponents for cocycles taking values on the non-abelian Lie group $\glr$  and applications to smooth dynamical systems. Given a linear cocycle $\A:\Sigma \to \glr$ over a two-sided subshift
of finite type $(\Sigma, T)$, $x\in X$ and $n\ge 1$ consider
$$
\mathcal{A}^{n}(x):=\mathcal{A}\left(T^{n-1} (x)\right) \ldots \mathcal{A}(x)
$$
and 
$\mathcal{A}^{-n}(x):=\mathcal{A}(T^{-n}(x))^{-1}\mathcal{A}(T^{-n+1}(x))^{-1} \dots \, \mathcal{A}(T^{-1}(x))^{-1}.$ 
For each
$T$-invariant probability measure $\mu$ such that $\log\|\mathcal A^{\pm 1}\|\in L^1(\mu)$, Oseledets' theorem ensures that for $\mu$-almost every $x$ there exist $k(x)\ge 1$ and an $\mathcal A$-invariant splitting $\mathbb R^d=E^1_x \oplus E^2_x \oplus \dots \oplus E_x^{k(x)}$ so that the \emph{Lyapunov exponents} are well defined by the limits
\begin{equation*}
    \lambda_i(\mathcal A,x)=\lim_{n\to\infty} \frac1n \log \|\mathcal A^n(x)v\|, \quad \text{for every $v\in E^i_x\setminus\{0\}$ }.
\end{equation*}
We denote by 
$ 
\chi_{1}(\mathcal{A}, x ) \geq \chi_{2}(\mathcal{A}, x )\geq \ldots \geq \chi_{d}(\mathcal{A}, x )
$  
the Lyapunov exponents, counted with multiplicity, of the cocycle $\mathcal A$ at $x\in \Sigma$. Set also 
$\chi_i(\mu, \A):=\int \chi_i(\mathcal{A}, \cdot ) d\mu $  for $1\le i \le d.$
We note that the Lyapunov exponents $\chi_i(\mathcal A,x)$ can be computed as 
$$
\chi_i(\mathcal A,x)
= \lim_{n\to \infty}\frac{1}{n}\log \sigma_{i}(\mathcal{A}^{n}(x))
$$
where $\sigma_{1}, \ldots, \sigma_d$ are the singular values, listed in decreasing order according to multiplicity \cite{horn1994topics}.
 The \textit{Lyapunov spectrum} of the linear cocycle 
$\mathcal A$ is the set
\begin{equation}
\label{def:Lyapunov-spectrum}
\vec{L}=\bigg\{\vec{\alpha} \in \R^{d}: 
E(\vec{\alpha}) \neq \emptyset \bigg\}  \end{equation}
where,  for each
$\vec{\alpha}:=(\alpha_1, \ldots, \alpha_d) \in \R^{d}$, the $\vec{\alpha}$-\textit{level set} is
 defined by 
 \begin{equation}\label{eq:123defl}
E(\vec{\alpha})=\bigg\{ x\in \Sigma: \lim_{n\to \infty}\frac{1}{n}\log \sigma_{i}(\mathcal{A}^{n}(x))=\alpha_i \text{  for every } 1\le i \le d \bigg\}.   
 \end{equation}

The level sets in~\eqref{eq:123defl} are substantially harder to describe than the level sets for the top Lyapunov exponent of linear cocycles. Indeed, the behavior of the full vector of Lyapunov exponents for matrix cocycles is much more intricate than that of the top  Lyapunov exponent alone. We refer the reader to \cite{Mohammadpour-survey, climenhaga, Barreira} for further background and references. To overcome these difficulties, several works have focused on studying the sets $E(\vec{\alpha})$ for matrix cocycles that satisfy a strong condition called domination.  This condition is open but not generic and, in essence, requires a uniform gap between successive singular values of the matrices $\A^{n}(x)$, independent of the base point $x$ and the iterate $n$ (see Subsection \ref{subsecdominatedsubs} for a precise definition). \color{black} The level sets $E(\vec{\alpha})$ have been analyzed for certain restricted classes of linear cocycles $\A$
 satisfying additivity-type property, such as domination (see \cite{BG06, FH}).


\smallskip

First results on the multifractal formalism of the $\vec{\alpha}$-level sets for Lyapunov exponents in ~\eqref{eq:123defl}  for generic matrix cocycles have been obtained by the first-named author
in the special case of locally constant linear cocycles \cite{Moh22-entropy, Moh23}, where the analysis boils down to understanding products of matrices.  However, these results do not contribute to the broader goal of describing the multifractal formalism of Lyapunov exponents in smooth dynamics, as derivative cocycles are rarely locally constant.

\smallskip
We say that a H\"older continuous cocycle $\mathcal{A}:\Sigma \to \glr$ over a topologically mixing subshift of finite type $(\Sigma, T)$ is called \textit{fiber bunched} if 
\[
\|\mathcal{A}(x)\|\|\mathcal{A}(x)^{-1}\|<2^{\alpha}
\quad\text{for every $x\in \Sigma$},
\]
where the factor $2$ appears in view of the metric on the shift space.
The fiber-bunching assumption is an open condition in the space of Ho\"lder continuous cocycles, and it roughly means that the cocycle is nearly conformal.

\smallskip
In this paper, we restrict our attention to an open and dense class of cocycles within the space of H\"older continuous and fiber-bunched cocycles. More precisely, we consider the class of typical cocycles - introduced by Bonatti and Viana \cite{BV04} and Avila and Viana \cite{AV07-acta} - which is known to form an open and dense subset
within the set of fiber-bunched
cocycles \cite{AV07-acta,BV04,park20}. We will study the thermodynamic formalism of such cocycles, namely by describing the topological entropy of Lyapunov level sets.

\smallskip
The multifractal formalism of linear cocycles has been considered before, under some strong geometric constraints. More precisely, 
under a domination condition, the Lyapunov exponents vary smoothly with the cocycle, and the multifractal formalism can be reduced to that of Birkhoff averages (see e.g. \cite{Cuneo,BG06, FH, Moh22,Ruelle1979} and references therein). In the absence of a domination condition, the Lyapunov exponents are often discontinuous (cf. \cite{BochiViana2005}), and there is no reduction of the multifractal formalism for Lyapunov exponents to the additive framework. Note that typical cocycles need not have a domination condition (see e.g. \cite{Moh22}).
%
\color{black}
In order to overcome the lack of domination, one of the key novelties of this paper is the construction of special induced subsystems that resemble the construction of horseshoes for nonuniformly hyperbolic dynamical systems due to Katok \cite{Katok}.

A similar idea in the context of the higher-dimensional multifractal formalism has been previously considered by  \cite{BJKR, Moh22-entropy, feng-sh},
in the setting of locally constant cocycles.

\color{black}
This construction - of independent interest - combined with a relation between the entropy and the Lyapunov exponents of the original cocycle and the induced cocycle, allowed us to provide a full characterization of the topological entropy of Lyapunov level sets. Specifically, we establish a relative variational principle for the topological entropy of level sets and for the topological pressure of the singular value function, as well as a relative variational principle for the topological entropy of Lyapunov exponent level sets in the case of typical cocycles (see Theorems~\ref{thmA} and \ref{thmD} for the precise statements).
\color{black}

One of the key ingredients in the proof is that 
 typical cocycles are simultaneously quasi-multiplicative (see Proposition \ref{typical cocycles are qm}), which enables us to prove the existence of the limit of the topological pressure associated with the generalized singular value function (see Lemma \ref{topological_pressure}).

\smallskip
 Considering this class of fiber-bunched linear cocycles offers a significant advantage in terms of applications.
Indeed, if 
$f: M \rightarrow M$ is a smooth map or diffeomorphism of a closed Riemannian manifold $M$ 
and $E \subset T M$ is a $D f$-invariant sub-bundle,  the derivative map restricted to $E$ is a cocycle generated by the map $\mathcal{A}(x)=$ $D_x f: E_x \rightarrow E_{f(x)}$. 
 In the special case that
$f$ is uniformly expanding or uniformly hyperbolic such a cocycle can be modeled by a H\"older continuous cocycle over a subshift of finite type (see Section~\ref{sec: proofthmD} for details). Hence, building over the results for more general linear cocycles, we derive a multidimensional multifractal formalism for certain classes of Anosov diffeomorphisms  and repellers (see Theorems~\ref{thmE} and ~\ref{thmF} for the precise statements).

\smallskip
This paper is organized as follows. Section~\ref{sec:statements} presents the statement of the main results. In Section~\ref{prel}, we recall some preliminary results on symbolic dynamics, exterior powers, and topological entropy of invariant but not necessarily compact sets, and we define fiber bunched and typical cocycles in Section~\ref{sec:linearcocycles}. In Section~\ref{sec:proofupperbddthmA1}, we show that typical cocycles are simultaneously quasi-multiplicative and prove the upper bound of Theorem~\ref{thmA}. We introduce the concepts of domination and dominated subsystems in detail in Section~\ref{dominated_subsystem}. The proofs of the main results are presented in Sections~\ref{Sec: proof of theoremA}-\ref{sec:proofapp}. 

\section{Statement of the main results}
\label{sec:statements}

\subsection{Setting}
Given a transition matrix $Q \in \mathcal M_{k\times k}(\{0,1\})$ the one-sided
subshift of finite type determined by $Q$ is a  left shift map
$T : \Sigma_{Q}^{+} \to  \Sigma_{Q}^{+} $ defined as $T(x_n)_{n\in \N}$ = $(x_{n+1})_{n\in \N}$, acting on the space of sequences
\[
\Sigma_{Q}^{+}:=\Big\{x=(x_{i})_{i\in \N} : x_{i}\in \{1,...,k\} \hspace{0.2cm}\textrm{and} \hspace{0.1cm}Q_{x_{i}, x_{i+1}}=1 \hspace{0.2cm}\textrm{for all}\hspace{0.1cm}i\in \N\Big\}.
\]
Similarly, we define a two-sided subshift of finite type $(\Sigma_{Q}, T)$  on the space 
\[
\Sigma_{Q}:=\Big\{x=(x_{i})_{i\in \Z} : x_{i}\in \{1,...,k\} \hspace{0.2cm}\textrm{and} \hspace{0.1cm}Q_{x_{i}, x_{i+1}}=1 \hspace{0.2cm}\textrm{for all}\hspace{0.1cm}i\in \Z\Big\}.
\]
For notational simplicity, we shall omit the dependence of $Q$ on the space $\Sigma_Q$ and $\Sigma_{Q}^{+}$ and simply write $\Sigma$ and $\Sigma^{+}$, respectively.
We denote by $\mathcal{L}_n$ the set of all admissible words of length $n$ of $\Sigma_Q$ and write $\mathcal{L}=\bigcup_{n\ge 1} \mathcal{L}_n$. For each $I\in \mathcal L$ let $|I|$ denote the length of the word $I$ and let
$$
[I]=\bigg\{x\in \Sigma_Q \colon x_j=i_j \; \text{for every } \, 0\le j \le |I|-1\bigg\}
$$
be the cyclinder set in $\Sigma_Q$ determined by $I$.
The case where all entries of the matrix $Q$ are equal to 1 corresponds to the \textit{full shift}. 
We denote by $\mathcal{M}_{\text{inv}}(T)$ the space of $T$-invariant probability measures.

For any linear cocycle $\mathcal{A}: \Sigma \rightarrow \text{GL}(d,\R)$ and $\mathrm{I} \in \mathcal{L}$, we define
\begin{equation}\label{def-A(I)}
  \|\mathcal{A}(\mathrm{I})\|:=\max _{x \in[I]}\left\|\mathcal{A}^{[I]}(x)\right\|.
\end{equation}

\color{black}

\smallskip
In order to describe the level sets of Lyapunov exponents for linear cocycles we shall need a generalization of Falconer's singular value function.
Given $A\in \mathcal M_{d\times d}(\mathbb R)$,
Falconer's \textit{singular value function} 
$\varphi^{s}({A})$ $(s\ge 0)$ 
is defined by
$$
\varphi^{s}({A}) =
    \begin{cases}
    \begin{array}{ll}
       \sigma_{1}({A}) \cdots \sigma_{k}({A}) \sigma_{k+1}({A})^{s-k},  & \text{if $0\le k \le d-1$ and $s\in [k,k+1]$}  \\
        \det({A})^{\,\frac{s}{d}} & 
        \text{if $s\geq d$}
    \end{array}   
    \end{cases}.
$$
%
 Given a vector $q:=(q_{1}, q_2,\dots,  q_{d})\in \R^d$, we define the \textit{generalized singular value function} $\psi^{q_{1}, \ldots, q_{d}}: \mathcal M_{d\times d}(\mathbb R) \rightarrow[0, \infty)$ as
\begin{equation}
    \label{eq:defgenps}
\psi^{q_{1}, \ldots, q_{d}}({A}):=\sigma_{1}({A})^{q_{1}} \cdots \sigma_{d}({A})^{q_{d}}=\left(\prod_{i=1}^{d-1}\left\|{A}^{\wedge i}\right\|^{q_{i}-q_{i+1}}\right)\left\|{A}^{\wedge d}\right\|^{q_{d}}.
\end{equation}
When $s \in [0, d]$, the singular value function $\varphi^{s}(A)$ coincides with the generalized singular value function 
 $\psi^{q_{1}, \ldots, q_{d}}({A})$ where 
$$
\left(q_{1}, \ldots, q_{d}\right)=(\underbrace{1, \ldots, 1}_{m \text { times }}, s-m, 0, \ldots, 0)
\quad\text{and}\quad m=\lfloor s\rfloor.
$$
For notational simplicity we shall write $\psi^{q}({A}):=\psi^{q_{1}, \ldots, q_{d}}({A})$
for each $A\in \mathcal M_{d\times d}(\mathbb R)$. 

\begin{rem}
{ We should notice that even though some similarity between the previous expressions, Falconer's singular value function $\varphi^{s}(A)$ is sub-multiplicative whereas the generalized 
singular value function $\psi^{q}(A)$ is neither sub-multiplicative nor supermultiplicative.}    
\end{rem}


\medskip
The main goal of this paper is to contribute to the 
multifractal
formalism  of a wide class of $GL(d, \R)$-valued and H\"older continuous matrix cocycles over $\Sigma.$
In \cite{BV04}, Bonatti and Viana introduced the notion of \textit{typical cocycles} 
and showed that the set of typical cocycles forms an open and dense subset of the space of fiber-bunched H\"older continuous linear cocycles (see Section \ref{prel} for the definitions). 
Given a cocycle $\mathcal A : \Sigma \to GL(d,\mathbb R)$ and  $q\in \mathbb R^d$ we will consider 
the family of 
potentials 
\begin{equation}
    \label{eq:singularvfunctioncocycle}
\log \Psi^q(\mathcal A):=\Big(\, \log \psi^{q}(\mathcal{A}^{n}(\cdot)) \Big)_{n\ge 1},
\end{equation}

which is well defined because $\psi^{q}(\mathcal{A}^n(x))\ge 0$ for every $x\in \Sigma$. 
\color{black}


\smallskip
The \emph{topological pressure} of the generalized singular value potential~\eqref{eq:singularvfunctioncocycle} is defined by 
\begin{equation}\label{eq:defPressure}
P^*( \log \Psi^q(\mathcal{A}) ):=\limsup_{n \to \infty}\frac{1}{n} \log Z_n(q),     \hspace{0.5cm} \forall q \in\R^d,     
\end{equation}
where 
\begin{equation}
\label{eq:defpartitionfunction}
Z_{n}(q):=\sum_{I \in \mathcal{L}_n} \psi^q(\mathcal{A}(I))
\quad \text{and}
\quad
\psi^q(\mathcal{A}(I)):=\max_{x\in [I]} \psi^q(\mathcal{A}^{n}(x)).
    \end{equation}
The function $Z_{n}(q)$ is often referred to as 
the \emph{partition function}. 
In case $\mathcal{A}$ is a typical cocycle one can replace $\limsup$ by a limit in ~\eqref{eq:defPressure} and we denote it by $P(T,\log \Psi^q(\mathcal{A}))$ (cf. Section \ref{sec:proofupperbddthmA1}). 

\smallskip
\subsection{General statements}
The first main result guarantees that the topological entropy of vector-valued Lyapunov exponent spectra 
(recall that $\vec L$ stands for the Lyapunov spectrum of the linear cocycle (cf. ~\eqref{def:Lyapunov-spectrum}), which is closed and convex (cf. \cite{park20}))
can be computed using  a Legendre-type transform 
 of the topological pressure 
of the generalized singular value potential. 
More precisely:

 \begin{maintheorem}\label{thmA}
\emph{Let $\mathcal{A}:\Sigma \to \text{GL}(d,\R)$ be a typical cocycle.  Then, 
for each $\vec{\alpha}\in \mathring{\vec{L}}$, 
  \[
  \begin{aligned}
  h_{\mathrm{top}}(T,E(\vec{\alpha}))
  & = \inf_{q\in \R^d}
  \Big\{\; P \left( T, \log \Psi^{q}(\mathcal{A})  \right)- \langle q, \vec{\alpha} \rangle \; \Big\} \\
  & =
\sup\Big\{ \, h_{\nu}(T)  \colon \nu \in \mathcal M_{inv}(T), \; 
     \chi_i(\nu, {\mathcal A}) =\alpha_i, \forall 1\le i \le d\Big\}.
  \end{aligned}
  \]
  }
  \end{maintheorem}

The relative variational principle in Theorem \ref{thmA}  provides an affirmative answer to a problem raised by Breuillard and Sert  \cite[Problem (7)]{BS21}. 

\smallskip
A different problem would be to obtain a variational principle for the topological pressure function $\log \Psi^q(\mathcal{A})$ involving the space of $T$-invariant probability measures. 
Variational principles for 
sub-additive and for super-additive potentials have been obtained in \cite{CFH,CPZ}. However, $\log \Psi^q(\mathcal{A})$ is neither sub-additive nor super-additive. In case $\A$ is a typical cocycle and $\mu \in \M(T)$, $\lim_{n\to \infty} \frac{1}{n} \int \log \psi^{q}(\A^n(x)) d\mu(x)$ exists (cf. Remark \eqref{LE exists}). In the following theorem, we overcome the lack of such additive properties for $\log \Psi^q(\mathcal{A})$ and establish the following variational principle.

 \begin{maintheorem}\label{thmD}
\emph{Let $\A:\Sigma \to \glr$ be a typical cocycle.  Then,  the following variational principle  holds: for any $q \in \R^d$,
\[ 
P(T,\log \Psi^{q}(\A) )=\sup\bigg\{ h_{\mu}(T)+\lim_{n\to \infty} \frac{1}{n} \int \log \psi^{q}(\A^n(x)) \,d\mu(x): \mu \in \M(T) \bigg\}. 
\]
}
 \end{maintheorem}

 \smallskip
\subsection{Applications to smooth dynamical systems}

The previous results have a wide range of applications. In what follows we shall focus on the multifractal analysis for Anosov diffeomorphisms.  Recall that  $f$ is a $C^{1}$ Anosov diffeomorphism of a closed manifold $M$ if there exist a $D f$-invariant splitting $T M=E^s \oplus E^u$ and constants $C>0, \nu \in(0,1)$ such that 
$$
\left\|\left.D_x f^n\right|_{E_x^s}\right\| \leq C \nu^n \text { and }\left\|\left.D_x f^{-n}\right|_{E^u_x}\right\| \leq C \nu^n 
$$
for every $x\in M$ and $n \in \mathbb{N}$.
 The restrictions of the derivative cocycle of a $C^{1+\alpha}$ Anosov diffeomorphism to the stable and unstable subbundles can be modeled by a H\"older continuous cocycle over a subshift of finite type. Indeed,
for a $C^{1+\alpha}$ Anosov diffeomorphism $f$, the unstable and stable bundles $E^s$ and $E^u$ are $\beta$-H\"older continuous, for some $\beta \in(0, \alpha]$ (cf. \cite{HP}).
%
Denoting the dimension of the unstable bundle $E^u$ by $d$, one can realize $\left.D f\right|_{E^u}$ as a 
$\text{GL}(d,\R)$-cocycle over a suitable subshift of finite type $\left(\Sigma, \sigma\right)$. Indeed, the existence of a Markov partition for $f$ \cite{Bow} results in a H\"older continuous surjection $\pi: \Sigma \rightarrow M$ such that $f \circ \pi=\pi \circ \sigma$. By choosing a Markov partition of sufficiently small diameter, one  may assume that the image of each cylinder $[j]$ of $\Sigma, 1 \leq j \leq q$, is contained in an open set on which $E^u$ is trivializable. For $x \in[j]$, we let $L_j(x)$ : $\mathbb{R}^d \rightarrow E_{\pi (x)}^u$ be a fixed trivialization of $E^u$ over $\pi([j])$. 
We define the $\alpha$-H\"older $\text{GL}(d,\R)$-cocycle 
$\mathcal{A}$ over the subshift $\left(\Sigma, \sigma\right)$ by
\begin{equation}\label{definition of cocycle for Ansosov}
    \mathcal{A}(x):=\left.L_k(\sigma (x))^{-1} \circ D_{\pi (x)} f\right|_{E^u} \circ L_j(x),
    \text{whenever $\sigma (x) \in [k]$.}
\end{equation}
We will say the derivative cocycle $\left.D f\right|_{E^u}$ is \emph{fiber-bunched} if there exists $N \in \mathbb{N}$ such that
\begin{equation}
\label{eq:bunchingDF}
\left\|\left.D_x f^N\right|_{E_x^u}\right\| \cdot\left\|\left(\left.D_x f^N\right|_{E_x^u}\right)^{-1}\right\| \cdot \max \left\{\left\|\left.D_x f^N\right|_{E_x^s}\right\|^\beta,\left\|\left(\left.D_x f^N\right|_{E_x^u}\right)^{-1}\right\|^\beta \right\}<1 .
\end{equation}
When $\left.D f\right|_{E^u}$ is fiber-bunched, the canonical stable and unstable holonomies for the cocycle $\left.D f\right|_{E^u}$ converge and are $\beta$-H\"older continuous (see Subsection~\ref{subsec:fiberbunching} for the definition of holonomy maps). 

\begin{maintheorem}\label{thmE}
 \emph{Let $f$ be a transitive $C^{1+\alpha}$ Anosov diffeomorphism of a closed manifold $M$ such that $\left.D f\right|_{E^u}$ is fiber-bunched.  If $\left.D f\right|_{E^u}$ is a typical cocycle then $$\begin{aligned}
h_{\text{top}}(E(\vec{\alpha}))
&= \inf_{q \in \R^d}\bigg\{P(f,\log \Psi^{q}(\left.D f\right|_{E^u}) )- \langle \vec{\alpha}, q \rangle \bigg\}\\
&=\sup \bigg\{h_{\mu}(f): \mu \in \mathcal M_{\text{inv}}(f), \; \chi_{i}(\mu, \left.D f\right|_{E^u})=\alpha_i \text{ for }1\le i \le d \bigg\}.
\end{aligned}
$$
for every $\vec{\alpha} \in \mathring{L}$.
}
\end{maintheorem}

Some comments are in order. Both the bunching and typicality requirements in Theorem~\ref{thmE} are $C^1$-open conditions. Moreover, as the bunching condition ~\eqref{eq:bunchingDF} is satisfied by a $C^1$-open set of $C^{1+\alpha}$-diffeomorphisms sufficiently $C^1$-close to a $C^{1+\alpha}$ Anosov diffeomorphism whose unstable direction is conformal, there exist $C^1$-open sets of $C^{1+\alpha}$ Anosov diffeomorphisms satisfying the requirements of the theorem.

 \medskip
 Our main results also find applications to the 
 Lyapunov multifractal formalism of expanding repellers. 
Given a $d$-dimensional Riemannian manifold $M$ and a $C^1$ map $h$ on $M$, a compact $h$-invariant subset $\Lambda \subset M$ is a \emph{repeller} if
 \begin{enumerate}
 \item 
 there exists an adapted norm $\|\cdot\|$ and $\lambda>1$ such that
$$
\left\|D_x h(v)\right\| \geq \lambda\|v\|\qquad 
\text{for all $x \in \Lambda$ and $v \in T_x M$,}
$$
\item  there exists a bounded open neighborhood $V$ of $\Lambda$ such that
$$
\Lambda=\Big\{x \in V: h^n x \in V \text { for all } n \geq 0\Big\} .
$$
 \end{enumerate}
 Given such a repeller $\Lambda$ and $\alpha \in(0,1]$, we say that $\left.h\right|_{\Lambda}$ is \emph{$\alpha$-bunched} if
$$
\left\|\left(D_x h\right)^{-1}\right\|^{1+\alpha} \cdot\left\|D_x h\right\|<1
\qquad 
\text{for all $x \in \Lambda$.}
$$

\begin{rem}
A natural class of $\alpha$-bunched repellers are small perturbations of conformal repellers. 
The 1-bunching assumption for repellers was studied first in \cite{falconer94}.    
\end{rem}

The next theorem provides a variational principle for the entropy of the Lyapunov spectrum level sets of most $\alpha$-bunched repellers.  More precisely:

\begin{maintheorem}\label{thmF}
 \emph{Let $M$ be a Riemannian manifold, $r>1$ and $h: M \rightarrow M$ be a $C^r$ map. Suppose $\Lambda \subset M$ is an $\alpha$-bunched repeller defined by $h$ for some $\alpha \in(0,1)$ with $r-1>\alpha$.  Then there exist a $C^1$-neighborhood $\mathcal{V}_1$ of $h$ in $C^r(M, M)$ and a $C^1$-open and $C^r$-dense subset $\mathcal{V}_2$ of $\mathcal{V}_1$ such that  for every $g \in \mathcal{V}_2$ and $\vec{\alpha} \in \mathring{\vec{L}}$,
 $$
 \begin{aligned}
h_{\text{top}}(E(\vec{\alpha}))
&= \inf_{q \in \R^d}\bigg\{P(g,\log \Psi^{q}(\left.D g\right|_{\Lambda_g}) )- \langle \vec{\alpha}, q \rangle \bigg\}\\
&=\sup \bigg\{h_{\mu}(g): \mu \in \M(g), \chi_{i}(\mu, \left.D g\right|_{\Lambda_g})=\alpha_i \text{ for }1\le i \le d \bigg\}.
\end{aligned}
$$
}
 \end{maintheorem}

We observe that the open set $\mathcal V_2$ in the previous theorem is formed by maps that generate typical repellers, whose definition will be given in Section~\ref{sec:proofapp}.

\section{Preliminaries}\label{prel}

In this section, we shall introduce some terminology and recall some basic concepts concerning subshifts of finite type, exterior powers, topological pressure and capacity of sets. The reader familiar with these topics may choose to skip this section, returning to it if necessary.

\subsection{Subshifts of finite type} 
Assume that $Q\in \mathcal M_{k\times k}(\{0,1\})$
is a transition matrix and $(\Sigma_Q,T)$ is the corresponding subshift of finite type.
It is well known that the primitivity of $Q$ (i.e. the existence of an integer $n\ge 1$ such that all the entries of $Q^n$ are positive) is equivalent to
the property of the subshift of finite type $(\Sigma_Q, T)$ to be \textit{topologically mixing}.
%
%
%
%
%
Endow the space $\Sigma=\Sigma_Q$ with the following metric $d$: for $x=(x_{i})_{i\in \Z}, y=(y_{i})_{i\in \Z} \in \Sigma$
\begin{equation}\label{metric}
d(x,y)= 2^{-\inf{\{k\ge 0} \colon \text{$x_{i}\neq y_{i}$ for some $|i| \le k$} \}},
\end{equation} 
The \textit{local stable set} at $x=(x_{i})_{i\in \Z}$ is the set
$
W_{\loc}^{s}(x)=\{(y_{n})_{n\in \Z} : y_{n}=x_{n} \hspace{0,2cm}\textrm{for all}\hspace{0.2cm} n\geq 0\} 
$
while the \textit{local unstable set}
at $x=(x_{i})_{i\in \Z}$ is defined by
$ W_{\loc}^{u}(x)=\{(y_{n})_{n\in \Z} : y_{n}=x_{n} \hspace{0,2cm}\textrm{for all}\hspace{0.2cm} n \leq 0\} .$ 
The global stable (resp. global unstable)  set of $x \in \Sigma$ is 
\[W^{s}(x):=\left\{y \in \Sigma: T^{n} y \in  W_{\loc}^{s}(T^{n}(x))\text { for some } n \geq 0\right\},\]
\[
(\text{resp.} \quad W^{u}(x):=\left\{y \in \Sigma: T^{n} y \in W_{\loc}^{u}(T^{n}(x)) \text { for some } n \leq 0\right\}).
\]
The two-side subshift of finite type  $(\Sigma,T)$ 
equipped with the metric $d$ in \eqref{metric}
is a hyperbolic homeomorphism (see \cite[Subsection 2.3]{AV10} for more details) and, in particular, it has a local product structure defined by
\begin{equation}
    \label{eq:prodst}
[x, y]:=W_{\loc}^{u}(x) \cap W_{\loc}^{s}(y)
\end{equation}
for any $x=(x_{i})_{i\in \Z}, y=(y_{i})_{i\in \Z} \in \Sigma$ so that $x_{0}=y_{0}$. 

 \subsection{Multilinear algebra}\label{wedge_product}
Given $A \in \glr$ and an integer $1\le l \le d$ 
recall that the \emph{$l^{th}$ exterior power of $A$} is the invertible linear map $A^{\wedge l} : \land^{l} \R^{d} \rightarrow \land^{l} \R^{d}$ satisfying
\[ 
A^{\wedge l}(e_{i_{1}}\wedge e_{i_{2}} \wedge ... \wedge e_{i_{l}})
= Ae_{i_{1}}\wedge Ae_{i_{2}} \wedge ... \wedge Ae_{i_{l}},
\]
where $e_i$'s are the standard vectors in the classical orthonormal basis of $\R^d.$
The linear map $A^{\wedge l}$ can be represented by a $\binom dl \times \binom dl$
real valued matrix whose entries are the $l \times l$ minors of $A$ (see e.g. \cite{HK}). 
In particular, for each $1\le l \le d$, the matrix $A^{\wedge l}$ has 
\begin{equation}
\label{eq:dimwed}
d_\ell:= \binom dl    
\end{equation} 
\color{black}
generalized eigenvalues.
It can also be shown that 
\begin{equation}\label{eq:norms.singular.values}
(AB)^{\wedge l}=A^{\wedge l} B^{\wedge l} \quad \text{and } \quad \|A^{\wedge l}\|=\sigma_{1}(A) \dots \sigma_{l}(A),  
\end{equation}
where $\sigma_{1},...,\sigma_{d}$ are the singular values of the matrix $A$. 
The right hand-side in ~\eqref{eq:norms.singular.values}
is extremely useful and allows to relate the potential $\log \Psi^q(\mathcal A(\cdot))$ with singular values and products of norms of exterior powers.

\subsection{Sequences of potentials}

Let $T: X \to X$ be a continuous map an let $\Phi=\{ \phi_n\}_{n\ge 1}$  be a sequence of continuous potentials $\phi_n : X \to \mathbb R$. The sequence $\Phi$ is \emph{additive} (with respect to $T$) if 
$\phi_{n+m}(x)=\phi_{n}(x) +\phi_{m}(T^{n}(x))$ for every $m,n\ge 1$ and every $x\in X$. If this is the case, a simple computation shows that 
$\phi_n=\sum_{j=1}^n \phi_1\circ T^j$
for every $n\ge 1$.
The sequence of potentials $\Phi=\{ \phi_n\}_{n\ge 1}$ is \textit{sub-additive} (with respect to $T$) if 
\[ \phi_{n+m}(x) \leq \phi_{n}(x) +\phi_{m}(T^{n}(x)) \quad \forall x\in X, \, \forall m,n \ge 1,\] 
and it is called
\textit{super-additive} if $-\Phi=\{ -\phi_n\}_{n\ge 1}$ is sub-additive.

Finally, the sequence $\Phi=\{\phi_{n}\}_{n\ge 1}$ is said to be an \textit{almost additive} (with respect to $T$) if there exists a constant $C > 0$ such that for any $m,n \ge 1$, $x\in X$, we have
\[
\phi_{n}(x)+\phi_{m}(T^{n})(x) -C \leq \phi_{n+m}(x)\leq \phi_{n}(x) +\phi_{m}(T^{n}(x))+C.
\]

\color{black}

\subsection{Topological pressure and entropy}

In the context of matrix cocycles it appears naturally some sequences of sub-additive real valued cocycles. In what follows we recall some notions from the thermodynamic formalism of non-additive sequences of potentials, including notions of entropy and pressure using Caratheodory structures.

\subsubsection{Topological pressure}

Let $(X, d)$ be a compact metric space and $T:X\rightarrow X$ be a continuous map. For any $n\in \N$, we define a metric $d_{n}$ on $X$ by
\begin{equation}\label{new_metric}
 d_{n}(x, y)=\max\Big\{\, d(T^{k}(x), T^{k}(y)) : 0\le k \le n-1\,\Big\}.
\end{equation}
Given $\varepsilon>0$ we say that $E \subset X$ 
is an $(n,\varepsilon)$-\textit{separated  subset}  if $d_{n}(x,y)\ge  \varepsilon$ 
for any two points $x\neq y \in E$.
\color{black}

Let $\Phi=\{\phi_{n}\}_{n=1}^{\infty}$ be a  sub-additive potential over $(X, T)$.  Given $\vep>0$ and $n\ge 1$ define
\[ P_{n}(T, \Phi, \varepsilon)=\sup \bigg\{\sum_{x\in E} e^{\phi_{n}(x)} : E \hspace{0,1cm}\textrm{is} \hspace{0,1cm}(n, \varepsilon) \textrm{-separated subset of }X \bigg\}.\]
The $\textit{topological pressure}$ of $\Phi$ is defined by
\begin{equation}\label{epsilon}
P(T,\Phi)=\lim_{\varepsilon \rightarrow 0}
\limsup_{n \to +\infty} \frac{1}{n} \log P_{n}(T, \Phi, \varepsilon),
\end{equation} 
where the limit in $\vep$ exists by monotonicity.
%
In \cite{CFH}, Cao, Feng and Huang established a variational principle for the topological pressure of sub-additive families of continuous potentials:
\begin{equation}\label{varitional}
P(T,\Phi)=\sup \bigg\{h_{\mu}(T)+\chi(\mu, \Phi): \mu \in \mathcal{M}_{\text{inv}}(T) \bigg\},
\end{equation}
where $h_{\mu}(T)$ is the measure-theoretic entropy and 
\begin{equation}\label{defchimu}
    \chi(\mu, \Phi):=\lim_{n\to \infty}\frac1n \int \phi_{n}(x) d\mu(x).
\end{equation}
Any invariant measure $\mu \in \mathcal{M}_{\text{inv}}(T)$ achieving the
supremum in \eqref{varitional} is called an \textit{equilibrium state} of $\Phi$.

\begin{rem}
By  upper semi-continuity of the functional ~\eqref{defchimu} with respect to $\mu$, in the weak$^+$ topology, one concludes that if the entropy map $\mu \mapsto h_{\mu}(T)$
is upper semi-continuous (e.g. in case $T$ is a subshift of finite type) then there exists at least one equilibrium state for each sub-additive family of continuous potentials $\Phi$. 
\end{rem}

\begin{rem}\label{(n, 1)separated sets}
If $(\Sigma, T)$ is a subshift of finite type 
endowed with the metric $d$ defined in ~\eqref{metric} then $T$ is an expansive map with expansivity constant 1, meaning that for any $x\neq y\in \Sigma$ there exists $n\in \mathbb Z$ so $d(T^n(x), T^n(y))\ge 1$. In particular the pressure $P(T,\Phi)$ coincides with 
\[
 P^*(\Phi)  =\limsup _{n \rightarrow \infty} \frac{1}{n} \log \sup \left\{\sum_{x \in E}  e^{\phi_{n}(x)}:E \hspace{0,1cm}\textrm{is} \hspace{0,1cm}(n, 1) \textrm{-separated subset of }\Sigma \right\},
\]
meaning that 
one can drop the limit in $\varepsilon$ from the definition of the pressure \eqref{epsilon} (see e.g. \cite{Walters}).
\end{rem}

\subsubsection{Topological entropy of subsets}\label{top_entropy}

Let $T$ be a continuous map on a compact metric space $X$. For any $n\in \N$ and $\varepsilon>0$ we define \textit{Bowen ball} $B_{n}(x, \varepsilon)$ as follows:
\[ 
B_{n}(x, \varepsilon)=\Big\{y\in X : d_{n}(x, y)<\varepsilon\Big\},
\]

where $d_n$ is the metric defined in ~\eqref{new_metric}.
\color{black}

Fix a subset $Y \subset X$ and $\varepsilon>0$. We define a covering of $Y$ as a countable collection of dynamic balls $\mathcal{Y}:=\left\{B_{n_{i}}\left(y_{i}, \varepsilon\right)\right\}_{i}$ such that 
$Y \subset \bigcup_{i} B_{n_{i}}\left(y_{i}, \varepsilon\right)$. Given a collection $\mathcal{Y}=\left\{B_{n_{i}}\left(y_{i}, \varepsilon\right)\right\}_{i}$, we define $n(\mathcal{Y})$ as the minimum value of $n_i$ among all indices $i$.  Let $s\geq 0$ and define
\[ S(Y, s, N, \varepsilon)=\inf \sum_{i} e^{-sn_{i}},\]
 where the infimum is taken over all collections $\mathcal{Y}=\{B_{n_{i}}(x_{i}, \varepsilon)\}_{i}$ that cover $Y$ and satisfy $n(\mathcal{Y})\geq N$. As $S(Y, s, N, \varepsilon)$ is non-decreasing with respect to $N$, the limit  $S(Y, s, N, \varepsilon)$ exists
 \[ S(Y, s, \varepsilon): = \lim_{N\rightarrow \infty} S(Y, s, N, \varepsilon).\]
 There is a critical value of the parameter $s$, which we denote by $h_{\text{top}}(T,Y, \varepsilon)$ such that
\[S(Y, s, \varepsilon)=\left\{\begin{array}{ll}
         0, & \mbox{$s>h_{\text{top}}(T,Y, \varepsilon)$},\\
        \infty, & \mbox{$s<h_{\text{top}}(T,Y, \varepsilon)$}.\end{array} \right . \] 
        Since $h_{\text{top}}(T,Y, \varepsilon)$ does not decrease with $\varepsilon$, the following limit exists,
        \begin{equation}
            \label{eq:defentropy-limit-vep}
        h_{\text{top}}(T,Y)=\lim_{\varepsilon \rightarrow 0}  h_{\text{top}}(T,Y,\varepsilon).
        \end{equation}
        We say that $h_{\text{top}}(T, Y)$ is the \textit{topological  entropy}  of $T$ at $Y$ (we shall denote it simply by $h_{\text{top}}(Y)$ for notational simplicity). We denote $h_{\text{top}}(X, T)=h_{\text{top}}(T)$  (see e.g. \cite{pesin}). \color{black} Again, in case $T$ is subshift of finite type it is enough to take $\varepsilon = 1$ above.

\section{Bunching and typicality for linear cocycles}\label{sec:linearcocycles}

Through this section, we will set some notation, collect some necessary results on the existence of holonomies for fiber-bunched cocycles, and define a notion of typicality for linear cocycles. 
\color{black} We will always assume that 
$T:\Sigma \to \Sigma$ is a topologically mixing subshift of finite type,
where $\Sigma \subset \{1,2,\dots,k\}^{\Z}$ is the symbolic space, endowed with the metric $d$ defined in ~\eqref{metric}, and $T(x_{n})_{n\in \Z}=(x_{n+1})_{n\in \Z}$ for any $(x_{n})_{n\in \Z} \in \Sigma$.

\subsection{One-step cocycles}

The simplest example of a linear cocycles $\mathcal A: \Sigma \to \glr$ is a Locally constant one, meaning that for each $x\in \Sigma$ there exists an open neighborhood $\mathcal V_x\subset \Sigma$ of $x$ such that $\mathcal A(y)=\mathcal A(x)$ for every $y\in \mathcal V_x$.

A simple class  of such linear cocycles are the so-called \textit{one-step cocycles} 
defined as follows.   Given a $k$-tuple of matrices $\textbf{A}=(A_{1},\ldots,A_{k})\in \glr^{k}$ , we associate with it the locally constant map $\mathcal{A}:\Sigma \rightarrow \glr$ given by $\mathcal{A}(x)=A_{x_{0}}$. This means the matrix cocycle $\mathcal{A}$ depends only on the zero-th symbol $x_0$ of $(x_{l})_{l\in \Z}$. It is clear that 
$$
\mathcal A^n(x) = A_{x_{n-1}} \cdot \dots \cdot A_{x_{1}} \cdot A_{x_{0}}
$$
for any $x=(x_{n})_{n\in \Z}\in \Sigma$, and it models a random product of the finite collection of matrices.

\medskip

\color{black}

\subsection{Fiber bunching and holonomies}\label{subsec:fiberbunching}
 Given $\alpha>0$, a linear cocycle $\mathcal{A}:\Sigma \rightarrow GL(d, \R)$ over a topologically mixing subshift of finite type $(\Sigma, T)$ is $\alpha$-H\"older continuous function if there exists $C>0$ such that
\begin{equation}\label{hol}
 \|\mathcal{A}(x)-\mathcal{A}(y)\|\leq Cd(x,y)^{\alpha} \hspace{0,2cm} \quad \forall x,y \in \Sigma.
\end{equation}
We denote by $C^{\alpha}(\Sigma, GL(d, \R))$ the vector space of $\alpha$-H\"older continuous cocycles over the topologically mixing subshift of finite type $(\Sigma, T)$. 

The action of a cocycle $\mathcal A$ can be observed by the skew-product $F: X\times \mathbb R^k \to X\times \mathbb R^k$ given by 
\begin{equation}
    \label{eq:skewp}
F(x,v)=(T(x),\mathcal A(x)v), \qquad 
(x,v)\in X\times \mathbb R^k,
\end{equation}
as the n-th iterate of $F$ is $F^n(x, v)=(T^{n}(x), \mathcal{A}^{n}(x)v)$, for each $n\ge 1$.
\color{black}
We say that the cocycle $\mathcal{A} \in C^{\alpha}(\Sigma, GL(d, \R))$ over $T$ is \textit{fiber bunched} if 
\begin{equation}\label{fiber}
\|\mathcal{A}(x)\|\|\mathcal{A}(x)^{-1}\|<2^{\alpha}
\quad\text{for every $x\in \Sigma$}
\end{equation} 
(note that the constant $2$ appearing above comes from \eqref{metric}). In rough terms, the fiber-bunching condition ensures the skew-product $F$ in ~\eqref{eq:skewp} is partially hyperbolic (see \cite{BV04} for more details).
Let $C_{b}^{\alpha}(\Sigma, GL(d, \R))$ denote the space of $\alpha$-Holder continuous and fiber-bunched cocycles over $T$,
and note that $C_{b}^{\alpha}(\Sigma, GL(d, \R))$
is a $C^0$-open subset of $C^{\alpha}(\Sigma, GL(d, \R))$.


\begin{defn}\label{holonomy}
Given $\mathcal A\in C^{\alpha}(\Sigma, GL(d, \R))$,
$x\in \Sigma$ and $y\in W_{\loc}^{s}(x)$
the \textit{local stable holonomy} 
$H_{y \leftarrow x}^{s} \in GL(d, \R)$ is defined
by the limit (in case the limit exists)
$$
H_{y \leftarrow x}^{s}:=\lim _{n \rightarrow+\infty} \mathcal{A}^n(y)^{-1} \mathcal{A}^n(x).
$$
Similarly, the \emph{local unstable holonomy} $H_{y \leftarrow x}^{u}$ is likewise defined as (in case the limit exists)
$$
H_{y \leftarrow x}^{u}:=\lim _{n \rightarrow-\infty} \mathcal{A}^n(y)^{-1} \mathcal{A}^n(x)
$$
for any $x\in \Sigma$ and $y\in W_{\loc}^{u}(x)$.
\end{defn}

Some comments are in order. First, even though the stable and unstable holonomies depend on the linear cocycle $\mathcal A$, we shall omit its dependence on $\mathcal A$ whenever possible, for notational simplicity. Second, it follows from \cite{BV04} that stable and unstable holonomies exist for fiber-bunched linear cocycles, that stable holonomies satisfy
\begin{itemize}
\item[a)]$H_{x \leftarrow x}^{s}=Id$ and $H_{z \leftarrow y}^{s} \circ H_{y \leftarrow x}^{s}=H_{z \leftarrow x}^{s}$ for any $z,y \in W_{\loc}^{s}(x)$.
\item[b)] $\mathcal{A}(y)\circ H_{y \leftarrow x}^{s}=H_{T(y) \leftarrow T(x)}^{s}\circ \mathcal{A}(x).$
\item[c)] $(x, y, v)\mapsto H^s_{y\leftarrow x}(v)$ is continuous.
\end{itemize}
and similar properties hold for unstable holonomies (with $s$ and $T$ replaced by $u$ and $T^{-1}$, respectively).
Third, one can use item $b)$ above to extend the definition to the global stable holonomy $H_{y\leftarrow x}^{s}$ for $y\in W^{s}(x)$ not necessarily in $W_{\loc}^{s}(x)$:
\begin{equation}\label{extension of holonomy}
H_{y\leftarrow x}^{s}=\mathcal{A}^{n}(y)^{-1} \circ H_{T^{n}(y)\leftarrow T^{n}(x)}^{s}\circ \mathcal{A}^{n}(x),
\end{equation} 
where 
$n\in \N$ is large enough such that $T^{n}(y)\in W_{\loc}^{s}(T^{n}(x))$ (the global unstable holonomy can be defined similarly).
Finally, the canonical 
holonomies vary 
$\alpha$-H\"older 
continuously  (see \cite{KS}), which means that there exists a constant $C > 0$ such that 
\begin{itemize}
\item[d)] $\|H_{y \leftarrow x}^{s} - I\| \leq C d(x,y)^{\alpha}$ for every $y \in W_{\loc}^{s}(x)$ (and an analogous statement for unstable holonomies).
\end{itemize}

\begin{rem}
It is worth mentioning that stable and unstable holonomies always exist for one-step cocycles (see \cite[Proposition 1.2]{BV04} and \cite[Remark 1]{Moh22}).
\end{rem}

In the remainder of this subsection, we will introduce one of the key objects to be used in this paper.
Assume that $x, y \in \Sigma$, $x^{\prime} \in W_{\text {loc }}^{u}(x)$, and $y^{\prime}:=T^n(x') \in W_{\text {loc }}^{s}(y)$ for some $n \in \mathbb{N}$. We can describe these points as forming a path 
from $x$ to $y$ via $x'$ and $y'$, which can be represented as:
\begin{equation}\label{eqpath}
x \stackrel{W_{\text {loc }}^{u}(x)}{\longrightarrow} x^{\prime} \stackrel{T^{n}}{\longrightarrow} y^{\prime} \stackrel{W_{\text {loc }}^{s}(y)}{\longrightarrow} y    
\end{equation}
(see Figure~\ref{fig1} below).
\begin{figure}[htb]
    \centering
    \includegraphics[width=0.9\linewidth]{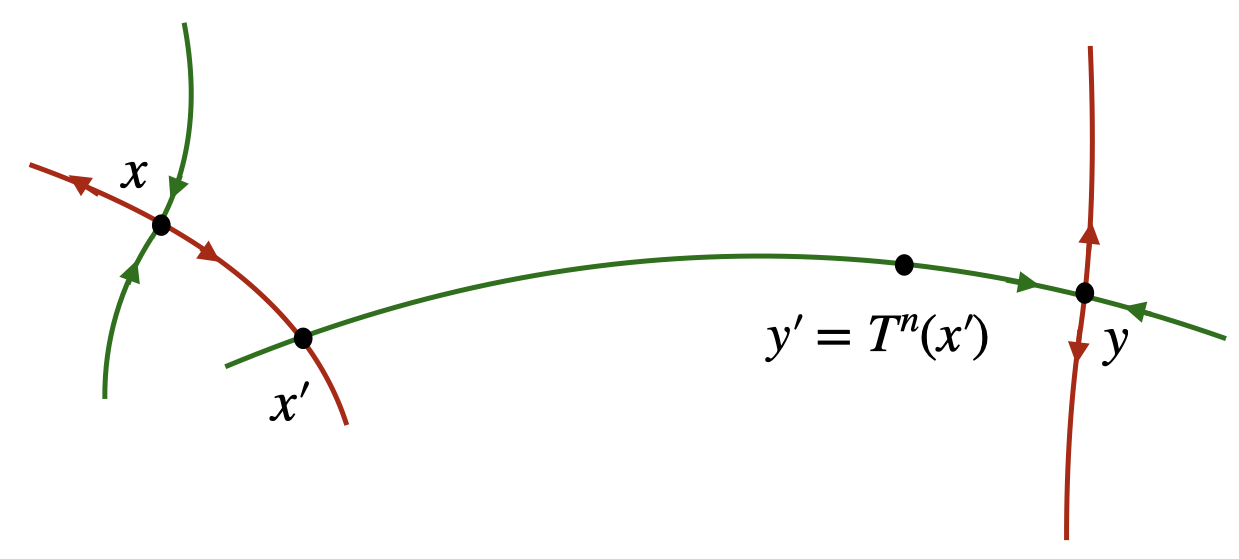}
    \caption{Path from $x$ to $y$ as in~\eqref{eqpath}.}
    \label{fig1}
\end{figure}
 To such a path and a linear cocycle $\mathcal{A}: \Sigma \rightarrow GL(d, \R)$ that admits canonical holonomies, one associates the 
matrix 
\begin{equation}\label{eq:loop}
B_{x, x^{\prime},y^{\prime},  y}:=H_{y \leftarrow y^{\prime}}^{s} \mathcal{A}^{n}\left(x^{\prime}\right) H_{x^{\prime} \leftarrow x}^{u}.
    \end{equation}
   When the context is clear, we will simply use the terminology “path from $x$ to $y$” while referring  to $B_{x,y}$.

\subsection{Typical cocycles}
\label{subsec:typical}
Let $T:\Sigma \to \Sigma$ be a topologically mixing subshift of finite type.
 Suppose that $p\in \Sigma$ is a periodic point of $T$. We say $z\in \Sigma\setminus\{p\}$ is a \textit{homoclinic point} associated to $p$ if 
 $z\in W^{s}(p) \cap W^{u}(p)$. We denote the set of all homoclinic points of $p$ by
$\mathcal{H}(p)$. For each 
$z\in W_{\text{loc}}^{s}(p) \cap W_{\text{loc}}^{u}(p)$
consider the matrix 
\begin{equation}\label{eq:matrixtildeH}
 {\widetilde H}_{p}^z :=H_{p \leftarrow z}^s \circ H_{z \leftarrow p}^u    
\end{equation}
associated to the homoclinic loop. Such matrices have been used in a crucial way to prove simplicity of the Lyapunov spectrum for typical fiber-bunched linear cocycles (cf. \cite{BV04}).
In general, given $z\in \mathcal H(p)$, up to replacing $z$ by some backward iterate, we may suppose that $z\in W_{\text{loc}}^{u}(p)$ and $T^{n}(z)\in W_{\text{loc}}^{s}(p)$ for
some $n \geq 1$. 
Then, using 
\eqref{extension of holonomy}, 
one defines
\begin{equation}
\label{eq:analogueeq}
 {\widetilde H}_{p}^z =\mathcal{A}^{-n}(p)\circ H_{p \leftarrow T^{n}(z)}^{s} \circ \mathcal{A}^{n}(z) \circ H_{z \leftarrow p}^{u}.    
\end{equation}

\begin{defn}\label{typical1}
Suppose that $\mathcal{A} \in C_{b}^{\alpha}(\Sigma, GL(d, \R))$. We say that $\mathcal{A}$ is \textit{1-typical} with respect to a pair $(p,z)$, where $p$ is a  periodic point for $T$ and $z\in \mathcal H(p)$ if:
\begin{itemize}
\item[(i)]  (pinching) \color{black} 
the eigenvalues of  $\mathcal{A}^{per(p)}(p)$ have multiplicity $1$ and distinct absolute value,
\item[(ii)] (twisting)   the eigenvectors $\left\{v_{1}, \ldots, v_{d}\right\}$ of $\mathcal{A}^{per(p)}(p)$ are such that, for any $I, J \subset \{1, \ldots, d\}$ with $|I|+$ $|J| \leq d$, the set of vectors
$$
\left\{ {\widetilde H}_{p}^z  \left(v_{i}\right): i \in I\right\} \cup\left\{v_{j}: j \in J\right\}
$$
is linearly independent.
\end{itemize}
\end{defn}

\begin{defn}
    We say $\mathcal{A} \in C_{b}^{\alpha}(\Sigma, GL(d, \R))$ is \textit{typical} if $\mathcal{A}^{\wedge t}$ is 1-typical with respect to the same pair $(p, z)$ for all $1 \leq t \leq d-1$.
\end{defn}


\begin{rem}\label{fixed point} For simplicity, we will always assume that $p$ in Definition~\ref{typical1} is a fixed point by considering the map $T^{\text{per}(p)}$ and the cocycle $\mathcal{A}^{\text{per}(p)}$ if necessary (this is possible because powers of typical cocycles are typical). Moreover, for any homoclinic point $z \in \mathcal{H}(p)$, $T^n (z)$ is a homoclinic point of $p$ for any  $n \in \mathbb{Z}$. 
\end{rem}

\begin{rem}\label{rem:typi}
The definition of typical cocycles described above is slightly stronger than the notion of typical cocycles introduced by Bonatti and Viana \cite{BV04}. In their definition, they only require 1-typicality of $\mathcal{A}^{\wedge t}$ for $1 \leq t \leq d / 2$, and they do not require the pair $(p, z)$ to be 
independent of $t$. 
 Although our version of typicality is slightly more demanding than the one in \cite{BV04}, a minor modification of their arguments shows that the set of typical cocycles remains open and dense
(see e.g. \cite[Remark 2.11]{park20} or \cite[page 1962]{Par22}).
\end{rem}

\subsection{Bounded distortion}
 We say that a family of continuous sub-additive real-valued potentials $\Phi:=\{\phi_{n}\}_{n=1}^{\infty}$ over  $(\Sigma, T)$ has \textit{bounded distortion} if  there exists $C\geq 1$ such that 
 \begin{equation}\label{bounded distortion}
C^{-1} \leq \frac{\exp(\phi_{n}(x))}{\exp(\phi_{n}(y))} \leq C
\end{equation} 
for any $x, y \in [I]$, every $I \in \mathcal{L}_{n}$ and every $n\in \N$.
Simple examples of sub-additive potentials with bounded distortion are given by singular value potentials $\varphi^{s}(\mathcal{A})$
 of one-step $GL(d, \R)$-cocycles $\mathcal{A}$, where $\varphi^{s}(\mathcal{A}^{n}(x))=\varphi^{s}(\mathcal{A}^{n}(y))$ for any $x, y \in [I].$


\section{Upper bound on the entropy of level sets }\label{sec:proofupperbddthmA1}

This section is devoted to the proof of the upper bound in the variational principle stated in Theorem~\ref{thmA}. More precisely, in this section we shall prove the following:

\begin{thm}
    \label{thmA-part1}
    If $\mathcal{A}:\Sigma \to \text{GL}(d,\R)$ is a typical cocycle then,
for each $\vec{\alpha}\in \mathring{\vec{L}}$, 
\begin{equation}
    \label{eq:ineq1thmA}
    h_{\mathrm{top}}(T,E(\vec{\alpha})) \le \inf_{q\in \R^k}
  \Big\{\; P \left(T,\log \Psi^{q}(\mathcal{A})\right)- \langle q, \vec{\alpha} \rangle \; \Big\}.
\end{equation}
\end{thm}

The proof of this theorem will occupy the remainder of this section. To this end,
we will start by proving that the pressure function $P^*\left(\log \Psi^{q}(\mathcal{A})\right)$, defined by a $\limsup$ in ~\eqref{eq:defPressure}, can be expressed as a limit in case of typical cocycles. 
Let us first recall some
necessary concepts. For any $\mathcal{A}:\Sigma \rightarrow GL(d, \R)$ and $I\in \mathcal{L}$, recall that 
\begin{equation*}\label{definition of the cocycle for words}
     \|\mathcal{A}(I)\|=\max_{x\in [I]} \|\mathcal{A}^{|I|}(x)\|
     \quad
     \text{and}
\quad 
\psi^{q}(\A(I))=\max_{x\in [I]}\psi^{q}(\A^{|I|}(x)).
\end{equation*}
Similarly, if $\mathbb{V}_{t}:=\mathbb{R}^{{\bf d}_t}$, where ${\bf d}_t=\left( \begin{array}{l}d \\ t\end{array}\right)$, and $\mathcal{A}_t:=\mathcal{A}^{\wedge t}: \Sigma \to \text{GL}(\mathbb{V}_{t})$ is the $t^{th}$-exterior power,  for each $1\le t \le d-1$,
 one defines 
\begin{equation}
    \label{eq:def-psiq}
\left\|\mathcal{A}_t(\mathrm{I})\right\|:=\max _{x \in[\mathrm{I}]}\left\|\mathcal{A}_t^{|\mathrm{I}|}(x)\right\| 
\quad
\text{and}
\quad 
\psi^{q}(\A_t(I)):=\max_{x\in [I]}\psi^{q}(\A_t^{|I|}(x)).
\end{equation}

We say that a fiber bunched linear cocycle $\mathcal{A}:\Sigma \to GL(d, \R)$  is \textit{simultaneously quasi-multiplicative} if there exist $c>0$ and $k \in \N$ such that for all $I, J \in \mathcal{L}$, there is $K=K(I, J) \in \mathcal{L}_{k}$ such that $IKJ \in \mathcal{L}$ and 
\begin{equation}
\label{eq:defquasimult}
   \|\mathcal{A}^{\wedge i} (IKJ)\|\geq c \|\mathcal{A}^{\wedge i}(I)\| \|\mathcal{A}^{\wedge i}(J)\|,
\quad \text{for every $1\le i \le d-1$}. 
\end{equation}

\begin{rem}\label{bdd distortion-remark}
If $\mathcal{A}:\Sigma \to \glr$ is a H\"older continuous and fiber-bunched cocycle over a topologically mixing subshift of finite type $(\Sigma, T)$ then the sub-additive sequence of potentials $\{\log \|\A_t^n \|\}_{n=1}^{\infty}$ satisfies the bounded distortion property $\eqref{bounded distortion}$ for every $1\le t \le d-1$ as a consequence of \cite[Lemma 3.10]{park20}.

\end{rem}

%
The next lemma ensures that the simultaneous quasi-multiplicativity condition is itself sufficient to guarantee that the pressure function is computed as a limit.

\begin{lem}\label{topological_pressure}
Assume that a fiber bunched cocycle $\mathcal{A}:\Sigma \to GL(d, \R)$ is simultaneously quasi-multiplicative. Then the sequence in the definition of  $P^*(T,\log \Psi^q(\mathcal{A}))$ is convergent, for any $q\in \R^{d}.$
\end{lem}

\begin{proof}
By the simultaneous quasi-multiplicativity property of $\mathcal{A}$, there exist $c>0$ and $k \in \N$ such that for all  $m, n > k$,  $I \in \mathcal{L}_{n}$ and $J \in \mathcal{L}_{m}$, there is $K \in \mathcal{L}_{k}$ such that $IKJ \in \mathcal{L}$ and 
$\|\mathcal{A}^{\wedge i} (IKJ)\|\geq c \|\mathcal{A}^{\wedge i}(I)\| \|\mathcal{A}^{\wedge i}(J)\|,$ for each $1\le i \le d-1$. 
Now, by definition of the generalized 
singular value function 
\eqref{eq:defgenps}
and the bounded distortion property  (see Remark \ref{bdd distortion-remark}), 
for any $q=(q_1, \ldots, q_d)\in \R^{d}$
there exists $c_0>0$ such that
\begin{equation}
    \label{eq:aux1}
    c_{0}^{-1}\prod_{i=1}^{d} \|  \mathcal{A}^{\wedge i} (IKJ)\|^{t_{i}}\geq \psi^q(\mathcal{A}(IKJ)) \geq c_{0}\underbrace{\prod_{i=1}^{d} \|\mathcal{A}^{\wedge i} (IKJ)\|^{t_{i}}}_\text{($\star$)},
\end{equation}
where $t_{i}=q_{i}-q_{i+1},$ and $q_{d+1}=0$ for $1\leq i \leq d.$
If $t_{i}< 0$, then by  sub-multiplicativity of the norm of the cocycle $\mathcal A^{\wedge i}$,  there is $C_{0}>0$ 
(depending only on $\mathcal A$)
such that
\begin{equation}\label{sub-multiplicative}
 \|\mathcal{A}^{\wedge i} (IKJ)\|^{t_{i}} \geq C_{0}^{t_{i}} \|\mathcal{A}^{\wedge i} (I)\|^{t_{i}} \|\mathcal{A}^{\wedge i} (J)\|^{t_{i}}.
\end{equation}
If $t_{i}\geq  0$, the simultaneous quasi-multiplicativity of $\mathcal{A}$, ensures that 
\begin{equation}\label{quasi-mult_eq}
 \|\mathcal{A}^{\wedge i} (IKJ)\|^{t_{i}} \geq c^{t_{i}} \|\mathcal{A}^{\wedge i} (I)\|^{t_{i}} \|\mathcal{A}^{\wedge i} (J)\|^{t_{i}}.
\end{equation}
Altogether, Remark \ref{bdd distortion-remark} and equations \eqref{sub-multiplicative} and \eqref{quasi-mult_eq} imply that 
\[ (\star) \geq C_{1} \prod_{i=1}^{d} \|\mathcal{A}^{\wedge i}(I)\|^{t_{i}} \prod_{i=1}^{d} \|\mathcal{A}^{\wedge i}(J)\|^{t_{i}},\]
where $C_{1}:=C_{1}(c_0, C^{t_{i}}, c^{t_i}, C_{0}^{t_{i}})>0.$ 
The latter, combined with ~\eqref{eq:defpartitionfunction} and ~\eqref{eq:aux1}, guarantees that  
$\psi^q(\mathcal{A}(IKJ)) \geq C_2 \;\psi^q(\mathcal{A}(I))\; \psi^q(\mathcal{A}(J))$ 
where $C_2:=c_0\,C_1>0$ and, consequently,
\[Z_{n+k+m}(q)\geq C_{2} Z_{n}(q)Z_{m}(q).\]
Hence the sequence $a_{n}:= \log Z_{n}(q) +\log C_{2}$ satisfies
$a_{n+m+k} \geq a_{n}+a_{m}$ for all $m, n>k$. 
 The convergence of the sequence 
$(\frac{a_n}{n})_{n\ge 1}$ follows as in the proof
of Fekete's lemma (cf. \cite{Bow}).  
This completes the proof of the lemma.

\end{proof}

 Now, consider the potential $\Phi_{\mathcal A}:\Sigma \to \mathbb R^d$ defined by
$$
 \Phi_{\mathcal{A}}:=\left( \log \sigma_{1}(\mathcal{A}), \ldots, \log\sigma_{d}(\mathcal{A})\right),
$$
and note that, denoting by  $\langle \cdot, \cdot \rangle$ the usual inner product on $\mathbb R^d$,
one has that
 $$
 \langle q, \Phi_{\mathcal{A}} \rangle =
  \log \Psi^q(\mathcal{A}) 
 \qquad \text{for every $q\in \R^d$}.
 $$ We denote 
 \begin{equation}\label{eqpnz}
 P_{n}(T, \langle q, \Phi_{\mathcal{A}} \rangle, 1):=\sup \left\{\sum_{x\in E} \psi^q(\mathcal{A}^{n}(x)) : E \hspace{0,1cm}\textrm{is} \hspace{0,1cm}(n, 1) \textrm{-separated subset of }\Sigma \right\}.    
 \end{equation}
 According to Lemma~\ref{topological_pressure}, 
 $$
 P\left(T,\log \Psi^q(\mathcal{A})\right)
 =P(T,\langle q, \Phi_{\mathcal{A}} \rangle)=\lim_{n\to \infty} \frac{1}{n}\log P_{n}(T, \langle q, \Phi_{\mathcal{A}} \rangle, 1).
 $$
\medskip

In \cite[Theorem 4.1]{park20}, Park proved that 
typical cocycles satisfy a weaker notion of simultaneous quasi-multiplicativity, defined in a 
way that the connecting word $K$ appearing in ~\eqref{eq:defquasimult} is bounded. We need the following strengthened version. 



\begin{prop}\label{typical cocycles are qm}
 Let $\mathcal{A}:\Sigma \to GL(d, \R)$ be a typical cocycle. Then $\mathcal{A}:\Sigma \to GL(d, \R)$   is simultaneously quasi-multiplicative.
\end{prop}
\color{black}
\begin{proof}

The proof of the proposition is inspired by the proof of Theorem 4.1 in \cite{park20}, where it is shown that
if $\mathcal{A}:\Sigma \to GL(d, \R)$ is a typical cocycle then there exist $c>0$ and $k \in \N$ such that for all $I, J \in \mathcal{L}$, there is $K=K(I, J) \in \mathcal{L}$ with $|K|\le k$ such that $IKJ \in \mathcal{L}$ and 
\begin{equation}
\label{eq:defquasimultweakP}
   \|\mathcal{A}^{\wedge i} (IKJ)\|\geq c \|\mathcal{A}^{\wedge i}(I)\| \|\mathcal{A}^{\wedge i}(J)\|,
\quad \text{for every $1\le i \le d-1$}. 
\end{equation}
To prove the proposition one needs to show that 
~\eqref{eq:defquasimultweakP} holds with words $K$ of constant length $k$, that is with words in $\mathcal L_k$. 
We start by noticing that, 
if $\ell_0\ge 1$ is the integer given by
\cite[Lemma~4.13]{park20}, for any $\ell \ge \ell_0$
the transition word $K=K(I,J)$ can be chosen such that 
\begin{equation}
\label{eqqK}
    |K(I,J)|= 2m + 2\bar \tau + n +\hat n + 2\ell
\end{equation}
where $m=m(I,J)$, $n=n(I,J)$ and 
$\hat n=\hat n(I,J)$ are constants determined by invariance of cones
(see Lemmas 4.22 and 4.12 in \cite{park20}) and
 $\bar\tau$ is a constant given by primitivity of the subshift of finite type, 
 all of them independent of $\ell$. 
In particular, there exists $C\ge 1$ so that $|K(I,J)|\le C + 2\ell$
(cf. \cite[page~1983]{park20}). 

\medskip
We now show that the length of the connecting words $K\in \mathcal{L}$ above can be chosen uniform  
we use the same notations as in \cite{park20}, for the reader's convenience. 
Fix $k_0=C+3\ell_0$. Given $I,J \in \mathcal L$ choose 
$$
\ell(I,J)=\frac12[C+3\ell_0-2m(I,J) - 2\bar \tau - n(I,J) -\hat n(I,J)]
$$ 
which, by construction, satisfies 
$\ell(I,J) \ge \frac32 \ell_0$.
The argument described above guarantees that 
there exists $K=K(I,J)\in \mathcal L$
satisfying equations ~\eqref{eq:defquasimultweakP} and \eqref{eqqK} with $\ell=\ell(I,J)$, hence $|K(I,J)|=k_0$.
This completes the proof of the proposition.

\end{proof}

\begin{rem}
The lack of control over the length of the connecting word $K$ in quasi-multiplicativity is a limitation when studying important applications (cf. \cite{BT22, MP-uniform-qm}). Proposition \ref{typical cocycles are qm} is the first result extending  it beyond one-step cocycles. Indeed,
 Bárány and Troscheit \cite[Proposition 2.5]{BT22} proved that if a one-step cocycle $\A$ is fully strongly irreducible and proximal, then $\A$ is simultaneously quasi-multiplicative. Recently, Mohammadpour and Park \cite{MP-uniform-qm} generalized their result for the norm of linear cocycles under the ireducibility assumption (see \cite[Corollary 1.2]{MP-uniform-qm}).        
\end{rem}

 Given a fiber bunched linear cocycle  
 $\A: \Sigma \to \glr$,  $\vec{\alpha}=(\alpha_1, \ldots, \alpha_d) \in \mathring{\vec{L}}$ and $r>0$, consider the set
\begin{equation}
\label{eq:g-alpha}
G(\vec{\alpha}, n, r):=\bigg\{x \in \Sigma:\hspace{0.1cm}  \bigg|\frac{1}{m}  \log \sigma_{i}(\mathcal{A}^{m}(x))-\alpha_i \bigg|<\frac{1}{r} \text{  for all }1 \leq i \leq d \text{ and } m \geq n \bigg\}.    
\end{equation}
It is clear from the definition that, for any $r>0$,
\[ E(\vec{\alpha}) \subset \bigcup_{n=1}^{\infty}  G(\vec{\alpha}, n, r ).\]

 \begin{lem}\label{thm for qm cocycles}
 Assume that $\mathcal{A}:\Sigma \to GL(d, \R)$ is
 fiber bunched and simultaneously quasi-multiplicative, and that $G(\vec{\alpha}, n, r ) \neq \emptyset$  for some $\alpha \in \mathring{\vec{L}}$. Then, for any $q=\left(q_{1}, \cdots, q_{d}\right) \in \mathbb{R}^d$,
$$
h_{\mathrm{top}}\big(\,  T, G(\vec{\alpha}, n, r )\, \big) \leq P (T,\log \Psi^q(\mathcal{A}))-\sum_{i=1}^{d}\left(\alpha_i q_i-\frac{|q_i|}{r}\right).
$$
\end{lem}
\color{black}
\begin{proof}
Since $T$ is a subshift of finite type,  
 $h_{\mathrm{top}}
 \big(\,  T, G(\vec{\alpha}, n, r )\, \big)=h_{\mathrm{top}}
 \big(\,  T, G(\vec{\alpha}, n, r ),\, 1 \big)
  $
(recall Subsection~\ref{top_entropy}).
Thus, if
$s<h_{\mathrm{top}}
 \big(\,  T, G(\vec{\alpha}, n, r )\, \big)$
then
$$ 
\infty=S(G(\vec{\alpha}, n, r ), s, 1)=\lim _{N \rightarrow \infty} S(G(\vec{\alpha}, n, r ), s, N, 1) 
$$
and there exists $N_{0}\ge 1$ such that
$$
S(G(\vec{\alpha}, n, r ), s, N, 1) \geq 1, \quad \forall N \geq N_{0} .
$$
Now take $N \geq \max \left\{n, N_{0}\right\}$ and let $E\subset G(\vec{\alpha}, n, r )$ be a $(N, 1)$-separated subset of maximal cardinality. As $G(\vec{\alpha}, n, r )\subset \bigcup_{x \in E} B_{N}(x, 1)$, it follows that
\begin{equation}\label{seprated_set}
\# E \cdot \exp (-s N) \geq S(G(\vec{\alpha}, n, r ), s, N, 1) \geq 1 .
\end{equation}
Moreover, recalling ~\eqref{eq:g-alpha}, if $x \in G(\vec{\alpha}, n, r )$ then
$$
\sum_{i=1}^{d} q_{i} \log \sigma_{i}\left(\mathcal{A}^{N}(x)\right) \geq N \cdot \sum_{i=1}^{d}\left(\alpha_{i} q_{i}-\frac{\left|q_{i}\right|}{r}\right)
$$ 
for each $q\in \R^{d}$. 
Therefore, recalling ~\eqref{eqpnz},
$$
\begin{aligned}
P_{N}(T, \langle q, \Phi_{\mathcal{A}} \rangle, 1)& \geq \sum_{x \in E} \exp \left(\sum_{i=1}^{d} q_{i} \log \sigma_{i}(\mathcal{A}^{N}(x))\right)\\
& \geq \# E \cdot \exp \left(N \cdot \sum_{i=1}^{d} \left(\alpha_i q_i -\frac{|q_i|}{r} \right) \right).
\end{aligned}
$$
By \eqref{seprated_set} and Lemma~\ref{topological_pressure} 
\[
P_{N}(T, \langle q, \Phi_{\mathcal{A}} \rangle, 1) \geq \exp \left(N\left[s+\sum_{i=1}^{d}\left(\alpha_i q_i -\frac{|q_i|}{r}\right)\right] \right)
\]
In consequence, 
$$
P(T,\log \Psi^q(\mathcal{A})) 
= P( T, \langle q, \Phi_{\mathcal{A}} \rangle) \geq s+\sum_{i=1}^{d}\left(\alpha_i q_i-\frac{|q_i|}{r}\right).
$$
Since $s< h_{\mathrm{top}}
 \big(\,  T, G(\vec{\alpha}, n, r )\, \big)$ was chosen arbitrary, this finishes the proof of the lemma.
\end{proof}

As a consequence, we obtain the following:

\begin{cor}\label{proof of the upper bound for vectors close to alpha}
Assume $\mathcal{A}: \Sigma \to \glr$ is a 
typical cocycle.  For any $\vec{\alpha}\in \mathring{\vec{L}}$, suppose that $\bigcup_{|\vec{\beta}-\vec{\alpha}|<\varepsilon} E(\vec{\beta}) \neq \emptyset$, 
for some $\varepsilon>0$. 
 Then, for any $q = \left(q_{1}, \cdots, q_{d}\right) \in \mathbb{R}^d$
and $r<\frac1\varepsilon$, 

\[
h_{\mathrm{top}}\left(T,\bigcup_{|\vec{\beta}-\vec{\alpha}|<\varepsilon} E(\vec{\beta})\right) \leq P(T,\log \Psi^q(\mathcal{A})) - \sum_{i=1}^{d}\left(\alpha_i q_i - \frac{|q_i|}{r}\right).
\]

\end{cor}
\begin{proof}
By Proposition \ref{typical cocycles are qm}, the cocycle 
$\mathcal{A}$ is simultaneously quasi-multiplicative. Moreover, 
 for each $r<\frac1\varepsilon$,
 one has $\bigcup_{|\vec{\beta}-\vec{\alpha}|<\varepsilon} E(\vec{\alpha}) \subset \bigcup_{n=1}^{\infty}G(\vec{\alpha}, n, r)$ 
 and, consequently,
\[
\begin{aligned}
    h_{\mathrm{top}}\left(T,\bigcup_{|\vec{\beta}-\vec{\alpha}|<\varepsilon} E(\vec{\alpha})\right) \leq h_{\mathrm{top}}\left(T,\bigcup_{n=1}^{\infty}G(\vec{\alpha}, n, r)\right)
    = \sup_{n \geq 1} \; h_{\mathrm{top}}(T,G(\vec{\alpha}, n, r)).
\end{aligned}
\]
The corollary follows directly from the above and the statement of Lemma \ref{thm for qm cocycles}.
\end{proof}

\begin{proof}[Proof of Theorem~\ref{thmA-part1}]
Given a typical cocycle
$\mathcal{A}:\Sigma \to GL(d, \R)$ 
the inequality ~\eqref{eq:ineq1thmA}
is a direct consequence of
Proposition~\ref{typical cocycles are qm} combined with Lemma~\ref{thm for qm cocycles}.
\end{proof}

 \section{Dominated subsystems }\label{dominated_subsystem}

The idea to obtain a lower bound for the entropy of Lyapunov level sets involves the obtainance of dominated subsystems from the original cocycle, as defined below.  

\subsection{ Construction of induced dominated subsystems}\label{subsecdominatedsubs}

\begin{defn}
Let $X$ be a metric space. We say that a linear cocycle $\mathcal{A}: X \to 
\text{GL}(d,\R)$ is 
 \textit{dominated with index $i$} if there exist constants $C >1$, $0<\tau<1$ such that
\[\frac{\sigma_{i+1}(\mathcal{A}^{n}(x))}{\sigma_{i}(\mathcal{A}^{n}(x))}\leq C \tau^n, \hspace{0.2cm} \forall n\in \N, x\in X.\]
Moreover, the cocycle $\mathcal{A}$ is \textit{dominated} if it is dominated with index $i$ for all $1\le i \le d-1.$
\end{defn}

In the case that $X$ is compact, 
Bochi and Gourmelon \cite[Theorem~A]{BGO}  proved that the latter notion of the domination of index $i$ is equivalent to the original definition of a dominated splitting requiring an $\mathcal A$-invariant splitting $X\times \mathbb R^d=V\oplus W$ with $\dim W=i$ and such that there exists $k\ge 1$ for which  
$
\|\mathcal A^k(x)\mid_{V_x}\| \le \frac12 \|\mathcal (A^k(x)\mid_{W_x})^{-1}\|^{-1}
$
for every $x\in X$ (cf. \cite{HP}).

 \begin{rem}\label{dominated-wedge}
 According to the multilinear algebra properties, $\mathcal{A}$ is dominated with index $i$ if and only if $\mathcal{A}^{\wedge i}$ is dominated with index 1. Therefore, the cocycle $\mathcal{A}$ is dominated if and only if $\mathcal{A}^{\wedge i}$ is dominated with index $i$ for any $1\le i \le d-1.$
 \end{rem}

 We recall that $\mathbb{V}_t=\mathbb{R}^{\mathbf{d}_t}, \text { where } \mathbf{d}_t=\binom{d}{t}$ and denote by $\mathbb{P} (\mathbb{V}_t)$ its projective space. We define the \emph{angular metric} $\rho$ on $\mathbb{P}^{d-1}$ by
$ 
\rho(\mathrm{u}, \mathrm{v}):=\min \{\measuredangle(u, v), \measuredangle(u,-v)\},
$ 
for every $u,v\in \mathbb{P}^{d-1}$.
Given any set $\mathrm{S} \subseteq \mathbb{P}^{d-1}$ and $\delta>0$, we denote the $\delta$-neighborhood of $S$ by
$$
\mathcal{C}(\mathrm{S}, \delta):=\left\{v \in \mathbb{P}^{d-1}: \rho(v, \mathrm{~S}) \leq \delta\right\}.
$$
Let $\mathring{\mathcal{C}}(v, \varepsilon)$ denote the interior
of the cone $\mathcal{C}(v, \varepsilon)$.
For any $A \in \text{GL}(d,\R)$, we define
$$
\|A\|_\rho:=\sup _{\mathrm{u} \neq \mathrm{v}} \frac{\rho(Au, Av)}{\rho(\mathrm{u}, \mathrm{v})}.
$$

\smallskip
Given a typical cocycle $\mathcal{A}$ and $1\le t \le d-1$, for notational simplicity we shall write 
$$
H^{s, t}:=\left(H^{s}\right)^{\wedge t}
    \quad\text{and}\quad
H^{u, t}:=\left(H^{u}\right)^{\wedge t}
$$ 
for the stable and unstable holonomies of $\mathcal{A}_{t}$, respectively (here $H^s$ and $H^u$ are stable and unstable canonical holonomies of $\mathcal{A}$, respectively).  

\medskip
In what follows we say that $\mathfrak p=(p,x, \sigma^n(x),y)$ is a \emph{path} 
 (of length $n$) starting at a periodic point $p$ provided that $x\in W_{\loc}^{u}(p)$ and that $\sigma^n(x)\in W_{\loc}^{s}(y)$. In case $y=p$ we say that $\mathfrak p$ is a \emph{loop} starting at $p$.

\medskip
A fiber-bunched cocycle $\mathcal{A}$ is called \emph{transverse} if for any $x, y \in \Sigma$, any vector $v \in \mathbb{R}^d\setminus\{\vec 0\}$ and any  hyperplane $W \subset \mathbb{R}^d$, there exists a path $B_{x, y}$ from $x$ to $y$ such that $B_{x, y} v \notin W$. By the compactness of $\Sigma$ and $\mathbb{P}^{d-1}$, every transverse fiber-bunched cocycle is \emph{uniformly transverse}:  there exist $\varepsilon>0$ and $N \in \mathbb{N}$ such that the 
path  $B_{x, y}$ can be chosen to have its length at most $N$  and that 
$$B_{x, y} v \notin \mathcal{C}(W, \varepsilon).$$

 \medskip
\begin{thm}\label{dominated_typical}
Let $A\in C_b^\alpha(\Sigma, \text{GL}(d,\mathbb R))$ be typical with respect to a pair $(p,z)$, where $p$
is a fixed point for $T$ and $z$ is an homoclinic point.
There exists $\tau^{\prime}>0$ so that, for any $\tau_1 \leq \tau^{\prime}, \tau_2>0$, the following holds: there exists $K_{0} \in \N$  such that for any $x\in \Sigma$ and $n\in \N$, there exists 
$\omega:=\omega_{x}\in \Sigma$ 
and integers
$m_{1}:=m_1(\omega) \ge 1$ and 
$m_2:=m_2(\omega)\ge 1$ satisfying
 \begin{itemize}
 \item[(1)]
$\omega:=\omega_{x}\in W^u_{loc}(p)$, 
$T^{m_1}(\omega)\in [x]_n,$
and
$T^{m_1+n+m_2}(\omega)\in W^s_{loc}(p)$;
 \item[(2)] the cocycle over the loop $\mathfrak p=(p, \omega, T^{m_1+n+m_2}(\omega),p)$  defined by
 \begin{equation}\label{path-main1}
  B_{p, \omega, T^{m_1+n+m_2}(\omega), p}^{t}:= H_{p \longleftarrow T^{m_1+n+m_2}(\omega)}^{s, t}\mathcal{A}_t^{m_1+n+m_2}(\omega)H_{\omega \longleftarrow p}^{u, t}
 \end{equation}
  satisfies
\begin{equation}\label{invaraint cone11}
   B_{p, \omega, T^{m_1+n+m_2}(\omega), p}^t\mathcal{C}(v_{1}^{t}, \tau_1)\subset \mathring{\mathcal{C}}(v_{1}^{t}, \tau_2)  
 \end{equation}
for all $t \in\{1, \ldots, d-1\}$, where $v_1^{t}$ is the eigendirection $\mathcal{A}_t (p)$ corresponding to the eigenvalue of the largest modulus; and
 \item[(3)] $1\le m_i \leq K_{0}$, for $i=1,2.$
\end{itemize}
 \end{thm}

\begin{proof}
By assumption, the cocycle $\mathcal{A}_t:\Sigma \to GL(\mathbb{V}_t)$  is 1-typical cocycle with respect to the distinguished fixed point $p$ and its homoclinic point $z$, for all $1\le t \le d-1$. 
Since the argument resembles the construction in the proof of  \cite[Theorem 4.6]{Par22}, we just sketch the proof. 
Let $\mathbb{V}_t$ be a real vector space of
dimension $d_t \in \N$ equipped with an inner product.

Assume, without loss of generality, that $p \in \Sigma$  is a fixed point and that the homoclinic point $z \in \mathcal{H}(p)$ lies in $W_{\text {loc }}^u(p)$. Let $x \in\Sigma$ and $n \in \mathbb{N}$ be given.  We denote the eigenvectors of $P_t:=\mathcal{A}_t(p)$ by $\big\{v_1^{(t)}, \ldots, v_{d_t}^{(t)}\big\}$, listed in the order of decreasing absolute values for their corresponding eigenvalues and consider the hyperplanes 
$$\mathbb{W}_i^{(t)}:= \text{span}\left\{v_1^{(t)}, \ldots, v_{i-1}^{(t)}, v_{i+1}^{(t)}, \ldots, v_{d_t}^{(t)}\right\} \subset \mathbb{V}_t.$$
The twisting assumption on the holonomy loop $\tilde{H}_t:=\widetilde{H^{\wedge t}}_{p}^{z}$ (recall Definition~\ref{typical1}) ensures that all coefficients $c_{i, j}^{(t)}$ in the linear combination 
$$
\tilde{H}_t v_i^{(t)}=\sum_{\substack{ j=1}}^{d_t} c_{i, j}^{(t)} v_j^{(t)}
$$
are nonzero for any $ 1 \leq i ,  j \leq d_t$.  \color{black} In particular, by compactness of $\mathbb{P} (\mathbb{V}_t)$ we can choose $\delta_t>0$ depending only on $\mathcal{A}_t$ such that the projectivized matrix 
$\H_t:=\mathbb{P} (\tilde{H}_t)$ satisfies
$$
\H_t\left(\bigcup_{i=1}^{d_t} \mathcal{C}\left(v_i^{(t)}, \delta_t\right)\right) \subseteq\left(\bigcup_{i=1}^{d_t} \mathcal{C}\left(W_i^{(t)}, \delta_t\right)\right)^c
$$
where $W_i^{(t)}:=\mathbb{P}\left(\mathbb{W}_i^{(t)}\right)$. Set $\delta:=\min _{1 \leq t \leq \kappa} \delta_t$.  The pinching assumption on $P_t$ (cf. Definition~\ref{typical1})  ensures that, for any $\varepsilon>0$, there exists $\ell_t(\varepsilon)\ge 1$ such that 
\begin{equation}\label{definition ell}
\mathrm{P}_t^{\ell} \H_t\left(\bigcup_{i=1}^{d_t} \mathcal{C}\left(v_i^{(t)}, \delta\right)\right) \subseteq \mathcal{C}\left(v_1^{(t)}, \varepsilon\right)
\end{equation}
holds for every $1 \leq t \leq d-1$ and every $\ell\ge \ell_t(\varepsilon)$. 
Given $\vep>0$ take $\ell(\varepsilon):=\max _{1 \leq t \leq \kappa} \ell_t(\varepsilon)$.

\smallskip
For each $q\in [p]_1$, we define the \textit{rectangle} through $p$ and $q$ as

\[
R_q^{t}:=H_{p \leftarrow [p, q]}^{u,t} \circ H_{[p, q]\leftarrow q}^{s,t} \circ H_{q \leftarrow [q, p]}^{u, t} \circ H_{[q, p]\leftarrow p}^{s, t}.
\]
\color{black}
By H\"older continuity of the cocycle 
for any $\delta_t>0$, there exists an integer $m = m(\delta_t) \in \N$ such that  
$$
 \max \big\{d([p,q],p), d([p,q],q)\big\} 
\le 2^{-m(\delta_t)} 
\; \Rightarrow \;  
(R_q^t)^{ \pm 1} \mathcal{C}(\mathrm{S}, c) \cup \mathcal{C}\left((R_q^t)^{ \pm 1} \mathrm{~S}, c\right) \subseteq \mathcal{C}(\mathrm{S}, 
c+\delta_t
)
$$
for any $S \subset \mathbb{P}^{d-1}$ and $c > 0$ (cf. \cite[Lemma 2.2]{Par22}).
Fix  $m=\max_{ 1 \leq t \leq d } m(\delta_t / 3)$. 
Using the  gluing orbit property of $(\Sigma, T)$ (see e.g. \cite{Bv19}) one can find $y \in T^{-n} W_{\text{loc}}^u\left(T^n x \right)$ and $n(x) \in \mathbb{N}$ such that $|n(x)-n|\leq N_0$ is uniformly bounded (depending only on $m$), $T^{n(x)} y$ belongs to $W_{\text {loc}}^s(p)$ and satisfies $d(T^{n(x)} y, p) \leq 2^{-m}$.    

\smallskip
For each $1\leq t\leq d-1$  consider the matrix $g_t:=H_{p \leftarrow T^{n(x)}(y)}^s \, \mathcal{A}_t^{n(x)}(y)$. 
By \cite[Proposition 2.8]{Par22},
there exists $c=c(g_t)>0$ and a hyperplance
$ \mathbb{U}_{g_t}\subset \mathbb R^d$ 
so that 
$$
\|g_t\mid_{\mathbb P^{d-1}\setminus \mathcal C(\mathbb U_{g_t},\varepsilon)}\|_\rho <c
$$
 The simultaneous uniform transversality for typical cocycles, proven in \cite[Theorem 4.3]{Par22},  ensures that there
are $N_1 \in \N$ and $\varepsilon_1 > 0$ such that the following holds: taking $p, y \in \Sigma$, the vectors $v_1^{(1)}, \ldots, v_{1}^{(d-1)}$ and the hyperplanes $\mathbb{U}_{g_1}, \ldots, \mathbb{U}_{g_{d-1}}$, there
exists a path $B^{(t)}_{p,  y}:=B^{(t)}_{p,b, T^{n(y)}(b), y}$ from $p$ to $y$ via $b$ and $T^{n(y)}(b)$ of length $0\le n(y) \le N_1$ such that
$$
\rho\left(B^{(t)}_{p,y} v_1^{(t)}, \mathbb{U}_{g_t}\right) \geq \varepsilon_1
$$
for all $1 \leq t \leq d-1$.
Set $\mathrm{u}_t=\mathrm{g}_t B_{p, y}^{(t)} v_1^{(t)}$ for $1\leq t \leq d-1.$ By \cite[Lemma 4.22]{park20}, there is $0\le a \le N(\delta / 3)$ such that $$\mathrm{P}_t^a \mathrm{u}_t \in \bigcup_{i=1}^{d_t} \mathcal{C}\left(v_i^{(t)}, \delta / 3\right)$$ for every $t$.

\smallskip
 We set $\ell:=\ell(\varepsilon)$ as in \eqref{definition ell} and take $$\omega:=T^{-(a+n(x)+n(y)+\ell)}\left[T^{a+n(x)+n(y)+\ell} (b), z\right] \in W_{\text {loc }}^u(p).$$
  We can connect two paths $\mathrm{P}_t^a \mathrm{g}_t B_{p, y}^{(t)}$
and $P_{t}^{\ell}(\H_t)$ by \cite[Lemma 3.3]{Par22}. Then, we obtain the path $$\widetilde{\mathrm{B}}^{(t)}_{p, \omega,T^{a+n(x)+n(y)+\ell}(\omega), p}:= H_{p \longleftarrow T^{a+n(x)+n(y)+\ell}(\omega)}^{s, t}\mathcal{A}_t^{a+n(x)+n(y)+\ell}(\omega)H_{\omega \longleftarrow p}^{u, t}$$ along the path 
$$
p \xrightarrow{H^u} \omega \xrightarrow{T^{a+n(x)+n(y)+\ell}} T^{a+n(x)+n(y)+\ell}(\omega) \xrightarrow{H^s} p
$$
which, 
similarly to  \cite[Lemma 3.10]{Par22} 
can be shown to satisfy the following property: there exists $\tau^{\prime}>0$ such that for any $\tau_1 \leq \tau^{\prime}, \tau_2>0$, and $\xi>0$, there exists $\ell_1 \in \mathbb{N}$ such that for all $\ell \geq \ell_1$,

$$
\widetilde{\mathrm{B}}^{(t)}_{p, \omega,T^{a+n(x)+n(y)+\ell}(\omega), p} \mathcal{C}\left(v_1^{(t)}, \tau_1\right) \subseteq \mathcal{C}\left(v_1^{(t)}, \tau_2\right) 
$$
for all $1\le t\le d-1$. 
This completes the proof of items (1) and (2) in the theorem.

\medskip
We are left to prove item (3).
Since $a \leq N(\frac{\delta}{3}), n(y) \leq N_1$, and $|n(x)-n| \leq N_0$, the difference $\left|a+n(x)+n(y)+\ell-n\right|$ is bounded from above by $\ell+ N(\frac{\delta}{3})+N_0+N_1$.
Note that this upper bound depends only on the base dynamical system $(\Sigma, T)$  and the
cocycle $\A_t$ (that is, depends on the mixing rate of $T$ and the constant $\delta$ which depends on
the cocycle $\A$ but not on $x$ and $n$).
Thus, we choose $K_0:=2(\ell+N(\frac{\delta}{3})+N_0+N_1)$. 
This finishes the proof of the theorem.
\end{proof}

\begin{rem}\label{properties of the dominated path}
Observe that the matrix \eqref{path-main1} over the loop $\mathfrak p$ is given by the composition of an unstable holonomy map departing from $p$, an ($m_1+m_2+n$)-iterate of the cocycle $\mathcal A_t$ and a stable holonomy ending at $p$.
Moreover,
using the properties (a)-(d) of the canonical holonomies, property ~\eqref{extension of holonomy} and Theorem~\ref{dominated_typical} (1), one can write \eqref{path-main1}, in a way to highlight the intervention of the point $x\in \Sigma$ as
\begin{equation}\label{eqdefb}
 B_{p, \omega, T^{m_1+n+m_2}(\omega), p}^{t}= P_{2,t}  \;
H_{T^{n}(z_2)\longleftarrow T^{m_1+n}(\omega)}^{u, t}   H_{T^{m_1+n}(\omega)\longleftarrow T^{n}(x)}^{s, t} \, \mathcal{A}_{t}^{n}(x) \; P_{1,t}
\end{equation}
where $z_2:= [T^{m_{1}}(\omega), x]\in \Sigma$ (recall ~\eqref{eq:prodst} for the definition of the bracket) \; 
$$
P_{1,t} = H_{x \longleftarrow [x,T^{m_{1}}(\omega)]}^{u, t} H_{[x,T^{m_{1}}(\omega)] \longleftarrow T^{m_{1}}(\omega)}^{s, t} \mathcal{A}_{t}^{m_{1}}(\omega) H_{\omega \longleftarrow p}^{u, t}$$ 
and 
$$
P_{2,t} = H_{p \longleftarrow T^{m_2}(T^{n+m_1}(\omega))}^{s, t}\mathcal{A}_{t}^{m_2}(T^{n+m_1}(\omega))  H_{T^{m_1+n}(\omega)\longleftarrow T^{n}(z_2)}^{u, t}.
$$
Consider as well the path  $P_{i, t}^{-1}$ as the inverse of $P_{i, t}$, for $i=1,2$, and observe that, as a consequence of item (3) in Theorem \ref{dominated_typical}, there exists $K>0$ such that 
   \begin{equation}\label{paths is bounded}
 \max\Big\{\| P_{1, t}^{-1} \|, \|P_{2, t}^{-1}\|, \| P_{1, t} \|, \|P_{2, t} \| \Big\} \le K,\qquad 
 \text{for every $t \in\{1, \ldots, d\}$}.  \end{equation}
For notational simplicity, we denote the holonomies $H_{T^{n}(z)\longleftarrow T^{m_1+n}(\omega)}^{u, t},  H_{T^{m_1+n}(\omega)\longleftarrow T^{n}(x)}^{s, t}$ in ~\eqref{eqdefb} by $H_{1}^{t}$ and $H_2^{t}$, respectively.  
%

      
\end{rem}

\medskip
We are now in a position to construct dominated cocycles as subsystems of the typical cocycle $\mathcal{A}:\Sigma \to GL(d, \R)$.  By Theorem \ref{dominated_typical}, there exists $K_0 \in \N$ such that for all $n>2K_{0}$ and $x \in  \Sigma$, there exists $\omega_{x}\in \Sigma$ and a loop $\mathfrak p=(p, \omega, T^{m_1(\omega_x)+n+m_2(\omega_x)}(\omega),p)$ 
of length at most $n+2K_{0}$ such that it shadows the forward orbit
of $x$ up to time $n$ along its path, and such that \eqref{invaraint cone11} holds. 
Fix $n > 2K_0$ and collect all the finite words determined by the points $\omega_x \in \Sigma$ by taking
\begin{equation}\label{constract_dominated}
\mathbb A=\mathbb A_n:=\Big\{[\omega_x]_{m_1(\omega_x)+ n+ m_2(\omega_x)}^{w}:  x\in \Sigma  \Big\},
\end{equation} 
where $[\omega_x]_{m_1(\omega_x)+ n+ m_2(\omega_x)}^{w}$ stands for the finite word determined by the first $m_1(\omega_x)+ n+ m_2(\omega_x)$ symbols (from the zeroth coordinate) of $\omega_x$. All words in $\mathbb A$ have length bounded from above by 
$n+2K_0$, hence there are finitely many of them.
In particular
\begin{equation}\label{eq:cotaA}
 \#\mathbb A_n \le 
 \#  \Big( \bigcup_{\ell=0}^{2K_0} \mathcal L_{n+\ell} \Big)
 \le 2K_0 \, \# \mathcal L_{n+2K_0}.   
\end{equation}

\begin{cor}\label{definition of the dominated cocycle}
Let $\mathcal{A}:\Sigma \to GL(d, \R)$ be a typical cocycle and $\mathbb A$ be given by 
~\eqref{constract_dominated}. Then, there exists a one-step cocycle cocycle $\B: \mathbb A^{\Z} \to \glr$ over the full shift $(\mathbb A^{\Z}, f)$ that is dominated.
\end{cor}
\begin{proof}

For each finite word 
$a=(a_0, a_1, \dots, a_k) \in \mathbb{A}$, we assign the matrix 
$$
B_{p, \omega, T^{m_1+n+m_2}(\omega), p}:= H_{p \longleftarrow T^{m_1+n+m_2}(\omega)}^{s}\, \mathcal{A}^{m_1+n+m_2}(\omega)\, H_{\omega \longleftarrow p}^{u},
$$ 
where $\omega=\omega^a\in \Sigma$ is such that $\omega_i=a_i$ for every $0\le i \le k$ and $\omega_i=p_i$ otherwise (here $p=(p_i)_{i\in \mathbb Z}$ is the representation of the periodic point $p$). 
By construction, the cocycle $B_{p, \omega, T^{m_1+n+m_2}(\omega), p}^{\wedge t}=B_{p, \omega, T^{m_1+n+m_2}(\omega), p}^{t}$ satisfies the cone condition for each $1\le t\le d-1$, by Theorem \ref{dominated_typical} item (2). 

\smallskip

Consider the cocycle over a full shift $(\mathbb A^{\Z}, f)$
$$
\B: \mathbb A^{\Z} \to \glr
    \quad\text{given by}\quad 
    \B(a):=B_{p, \omega^{a_0}, T^{m_1+n+m_2}(\omega^{a_0}), p}.
$$
for every $a=(a_i)_{i\in \mathbb Z} \in \mathbb A^{\Z}$.
This is clearly a one-step cocycle since it depends only on the zeroth coordinate $a_0$ of $a\in \mathbb A^{\Z}$.

Since there exists a strongly invariant multicone for $\B^{\wedge t}$ for each $1\le t \le d-1$ by Theorem \ref{dominated_typical}, the cocycle $\B$ is dominated (cf. \cite[Theorem B]{BGO}). 
\end{proof}

 It is worth mentioning that Corollary \ref{definition of the dominated cocycle} extends a similar result from \cite{BJKR, Moh22-entropy}, which was previously established in the two-dimensional case and in higher dimensions for one-step cocycles.
Moreover, the first named author proved that if a cocycle 
 $\mathcal D$ 
is dominated with index 1 (characterized by the existence of invariant cone fields, or multicones)
then the potential 
$\left\{\log \|\mathcal{D}^{n}\|\right\}_{n=1}^{\infty}$ 
is almost additive. More precisely:

\begin{prop}[{\cite[Proposition 5.8]{Moh22}}]\label{prop:add}
Let $X$ be a compact metric space and let  $\mathcal{D}: X\rightarrow GL(d, \R)$  be a linear cocycle over a homeomorphism map $(X, T)$. Assume that the cocycle $\mathcal D$ is dominated with index 1. Then, there exists $\kappa>0$ such that for every $m,n>0$ and for every $x\in X$ we have
\[
||\mathcal{D}^{m+n}(x)|| \geq \kappa ||\mathcal{D}^m(x)|| \cdot ||\mathcal{D}^n(T^m(x))||.
\]
\end{prop}

\subsection{Approximation of pressure and metric entropy}

\vspace{.5cm}
From the previous construction, to each typical cocycle $\mathcal{A}:\Sigma \to GL(d, \R)$ and integer $n\ge 1$ one can associate a dominated cocycle 
$\mathcal{B}: \mathbb{A}^{\Z} \to \glr$ over a compact full shift $(\mathbb A^\Z,f)$ as in Corollary \ref{definition of the dominated cocycle}.  The \emph{topological pressure on the dominated sub-system} $\mathbb{A}=\mathbb{A}_n$ is defined by
the expression
\begin{equation}
\label{defPdominated}
P_{n, \mathcal{D}}(\log \Psi^{q} (\B)):=\lim _{k \rightarrow \infty} \frac{1}{k} \log \sum_{I_{1}, \ldots, I_{k} \in \mathbb{A}_n}\psi^{q} (\B(I_{1} \ldots I_{k})).    
\end{equation}

\smallskip

\begin{rem}
For each $n\in \mathbb N$ the expression ~\eqref{defPdominated} coincides with the topological pressure of the locally constant cocycle $\B$ over the full shift $f: \mathbb{A}^\Z\to \mathbb{A}^\Z$ (recall that the alphabet $\mathbb{A}$ depends on $n$) and the potential $\log \Psi^{q} (\B)$. Moreover,  
one can write the potential 
$\log \Psi^{q} (\B)$ using the potential $\Phi_{\mathcal{B}}: \mathbb{A}^{\mathbb Z} \rightarrow \mathbb{R}^k$ defined by
$  
\Phi_{\mathcal{B}}:=\left(\log \sigma_1(\mathcal{B}), \ldots, \log \sigma_d(\mathcal{B})\right),
$ 
observing that
$ 
\left\langle q, \Phi_{\mathcal{B}}\right\rangle=\log \Psi^q(\mathcal{B})$ for every 
$q \in \mathbb{R}^d$.     
\end{rem}

The following result guarantees that the topological pressure of a typical cocycle is well aproximated by the topological pressure of dominated sub-systems.

\begin{thm}\label{continuity_potential} If $\mathcal{A}:\Sigma \to GL(d, \R)$ is a typical cocycle then
\[\lim _{n \rightarrow \infty} \frac{1}{n} P_{n, \mathcal{D}}(\log \Psi^{q} (\B))=P(T,\log \Psi^{q}(\A)),\]
for each $q\in \R^d$ (uniformly on compact subsets).
\end{thm}

\begin{proof}
Fix $n\in \mathbb N$ and the finite alphabet $\mathbb A=\mathbb A_n$ determined by ~\eqref{constract_dominated}.
Applying Proposition \ref{prop:add} to the 1-step cocycle $\mathcal B: \mathbb A^\Z \to GL(d, \R)$, one concludes that there exists a constant $\kappa=\kappa_n>0$ such that 
\begin{equation}\label{almost-additivity}
   \|\mathcal{B}^{\wedge t}(IJ)\| \geq \kappa\, \|\mathcal{B}^{\wedge t}(I)\| \|\mathcal{B}^{\wedge t}(J)\|  
\end{equation}
for every $I, J \in \bigcup_{k=1}^{\infty}\mathbb{A}^{k}$ and  $1\le t \le d-1$.
We claim that the constant $\kappa$ can be chosen uniform for all $n$. Indeed, 
this follows from the fact that the constant $\kappa$ in 
Proposition~\ref{prop:add} is guaranteed by the existence of a constant invariant cone-field (recall equation~\eqref{invaraint cone11}), and that 
if $\mathcal C, \mathcal C_0\subset \mathbb R^d$ are one-dimensional cones such that $\overline{\mathcal C_0}\subset \text{int}(\overline {\mathcal C})$ then there exists $\kappa_0>0$ so that 
$\|Dv\| \ge \kappa_0\, \|D\|\,\|v\|$ for every $v\in C_0$ and every matrix $D\in \text{GL}(d,\R)$ so that $D(\mathcal C)\subset \mathcal C_0$
(cf. \cite[Lemma 2.2]{Bochi-Morris}).
 Therefore, there exists $C(\kappa)>0$ so that
$$
\begin{aligned}
 \frac{1}{n} P_{n, \mathcal{D}}(\log \Psi^{q} (\B)) & =  \lim _{k \rightarrow \infty} \frac{1}{n k} \log \sum_{I_1, \ldots, I_k \in \mathbb{A}} e^{\langle q, \Phi_{\mathcal{B}}(I_1 \ldots I_k) \rangle} \\
& \geq \lim_{k \rightarrow \infty} \frac{1}{n k} \log \sum_{I_{1}, \ldots, I_{k} \in \mathbb{A}}e^{\sum_{\ell=1}^{k}\langle q, \Phi_{\mathcal{B}}(I_{\ell}) \rangle 
- C( \kappa) k}  \\
&=\underbrace{\frac{1}{n}\log \sum_{I \in \mathbb{A}} e^{\langle q, \Phi_{\mathcal{B}}(I) \rangle}
-\frac{ C( \kappa) 
}{n}}_\text{(1)}.
\end{aligned}
$$

Recall that we denote $\A_t:=\A^{\wedge t}$ for each $1\le t \le d-1$. 
By the definition of the one-step cocycle $\B$  and Remark \ref{properties of the dominated path}, for any $2\le i \le d$, given $I \in \mathbb{A}$ 
in such a way that $[\omega_x]_{m_1(\omega_x)+ n+ m_2(\omega_x)}^{w}=I$ for some $x\in \Sigma$,
and given $a=(a_n)_{n\in \Z}$ so that $a_0=I$,  
$$\begin{aligned}
    \sigma_{i}(\mathcal{B}(I))& =\frac{\|\mathcal{B}^{\wedge i}(I)\|}{\|\mathcal{B}^{\wedge (i-1)}(I)\|} \\
    &=\frac{\|B_{p, \omega^{a_0}, T^{m_1+n+m_2}(\omega^{a_0}), p}^{\wedge i}\|}{\|B_{p, \omega^{a_0}, T^{m_1+n+m_2}(\omega^{a_0}), p}^{\wedge i}\|}\\
    &=\frac{ \|P_{1, i}H_1^i H_2^i \mathcal{A}_{i}^{n}(x)P_{2, i}\|}{ \|P_{1, i-1}H_1^{i-1} H_2^{i-1} \mathcal{A}_{i-1}^{n}(x)P_{2, i-1}\|},
\end{aligned}
$$
where $m_1=m_1(\omega^{a_0})$ and $m_2=m_2(\omega^{a_0})$ are the integers given by 
Theorem~\ref{dominated_typical}.

%

\smallskip
From the uniform continuity of the canonical holonomies $H_{y\longleftarrow x}^{*, t}$, with $*\in \{s,u\}$, there exists a uniform upper bound $C_{0}>1$ for the norm of such holonomies whenever 
$y \in W_{\mathrm{loc}}^{*}(x)$. Therefore, for any  $n \in \N$, $x \in \Sigma,$ and $1\le i \le d$,
 \begin{equation}\label{bounded for holonomy}
C_{0}^{-2} \|\mathcal{A}_{i}^{n}(x)\| \leq \|H_1^i H_2^i \mathcal{A}_{i}^{n}(x)\|.
 \end{equation}
Therefore, combining  \eqref{paths is bounded} and
\eqref{bounded for holonomy},
\[
\begin{aligned}
\|\mathcal{A}_{i}^{n}(x)\| 
    & \le C_{0}^{2} \; \|H_1^i H_2^i \mathcal{A}_{i}^{n}(x)\|
    = C_{0}^{2} \; \|P_{1, i}^{-1}P_{1, i} H_1^i H_2^i \mathcal{A}_{i}^{n}(x)P_{2, i}P_{2, i}^{-1}\| \\
    & \le C_{0}^{2} \; K^2 \|P_{1, i}H_1^i H_2^i \mathcal{A}_{i}^{n}(x)P_{2}^i\|.
\end{aligned}
 \]
Using once more \eqref{paths is bounded} and \eqref{bounded for holonomy}, we also obtain
\begin{equation*}\label{lower-bound-B}
 \|P_{1, i}H_1^i H_2^i \mathcal{A}_{i}^{n}(x)P_{2, i}\|  \leq C_{0}^{2} K^2 \|\mathcal{A}_{i}^{n}(x)\|.
\end{equation*}
Altogether we conclude that if $C':=C'(C_0, K)=\frac{1}{C_0^4 K^{4}}$ then  
\begin{equation}\label{ineqaultiy_singular_values}
\frac{1}{C'} \sigma_{i}(\mathcal{A}^{n}(x)) \geq \sigma_{i}(\mathcal{B}(I)) \geq C' \sigma_{i}(\mathcal{A}^{n}(x))
\end{equation}
for any $1\le i \le d$. 
Hence, there exists $C''>0$ so that
$  \psi^q(\mathcal{B}(I)) \geq C'' \psi^q(\mathcal{A}^{n}(x))
$ 
for each $q \in \R^d$, guarantees that 
$ 
(1)  \geq \frac{1}{n} \log \sum_{I\in \mathcal{L}_{n}}\psi^{q}(\A(I))-\frac{C(\kappa)}{n}+\frac{C''}{n},$ 
and ultimately guarantees that
\[
\liminf_{n \rightarrow \infty} \frac{1}{n} P_{n, \mathcal{D}}(\log \Psi^{q} (\B)) \ge P(T,\log \Psi^{q}(\A))
\]
for each $q\in \R^d$ (uniformly on compact subsets).
 \medskip

We proceed to show the converse inequality.  
First note that the continuous map 
 $$
 \pi: \mathcal L_{n+2K_0} \to \mathbb A_n 
 \quad \text{defined by}\quad 
\pi(\alpha_1, \alpha_2, \dots, \alpha_{n+2K_0})
=(\alpha_1, \alpha_2, \dots, \alpha_{n})
$$
is surjective and finite-to-one.
Noting that for each $J\in \mathbb A_n$  there exist finitely many elements $I \in \pi^{-1}(J)$ (recall \eqref{eq:cotaA}) 
and equations \eqref{almost-additivity} and ~\eqref{ineqaultiy_singular_values}
one concludes that
there exists $C'''>0$ so that 
$$
\begin{aligned}
 \frac{1}{n} P_{n, \mathcal{D}}(\log \Psi^{q} (\B)) & =  \lim _{k \rightarrow \infty} \frac{1}{n k} \log \sum_{I_1, \ldots, I_k \in \mathbb{A}_n} e^{\langle q, \Phi_{\mathcal{B}}(I_1 \ldots I_k) \rangle} \\
& \leq \lim_{k \rightarrow \infty} \frac{1}{n k} \log \sum_{I_{1}, \ldots, I_{k} \in \mathbb{A}_n}e^{\sum_{\ell=1}^{k}\langle q, \Phi_{\mathcal{B}}(I_{\ell}) \rangle + C( \kappa) k}  \\
&=\frac{1}{n}\log \sum_{I \in \mathbb{A}_n} e^{\langle q, \Phi_{\mathcal{B}}(I) \rangle}+\frac{ C( \kappa) 
}{n}\\
& \le 
\frac{1}{n}\log \sum_{J \in \mathcal{L}_{n+2K_0}} e^{\langle q, \Phi_{\mathcal{A}}(J) \rangle}+
\frac{ C''' }{n},
\end{aligned}
$$
 for every $q\in \mathbb R^d$.
By taking $n\to \infty$,
 $$
 P(T,\log \Psi^{q}(\A))\geq \lim_{n\to \infty}\frac{1}{n} \log P_{n, \mathcal{D}}(\log \psi^{q}(\B)). 
 $$ 
 This ends the proof of the theorem.
\end{proof}

\begin{rem}\label{LE exists}
Let $\A: \Sigma \to \glr$ be a fiber-bunched cocycle over a topologically mixing subshift of finite type $(\Sigma, T)$. Let $\mu \in \M(T)$. By the bounded distortion property (see Remark \ref{bdd distortion-remark}), for $q \in \R^d$ the observable $\psi^{q}(\A)$ is comparable to $\Pi_{i=1}^{d} \|\A^{\wedge i}\|^{t_{i}}$ up to some constant, where $t_i = q_{i} -q_{i+1}$ and $q_{d+1} = 0$ for $1\le i \le d.$ Since $\|\A^{\wedge i}\|^{ t_{i}}$ is either sub-multiplicative
or super-multiplicative (depending on the positivity of $t_i$) then the limit $\lim_{n \to \infty} \frac{1}{n}\int \log \|(\A^{\wedge i})^{n}(x)\|^{t_i} d\mu(x)$ does exist, for each $1\le i\le d$. Hence, the limit
\[\lim_{n \to \infty} \frac{1}{n}\int \log \psi^{q}(\A^{n}(x)) d\mu(x)\]
is well-defined.
\end{rem}

\begin{prop}\label{relation between entropies and LE}
Suppose that $\mathcal{A}:\Sig \to \glr$ is a typical cocycle.  There exists $K_0 \ge 1$ so that, for each $n>2K_0$ there is a dominated one-step cocycle 
$\mathcal{B}: \mathbb{A}^{\Z} \to \glr$ over a compact full shift $(\mathbb A^\Z,f)$ (depending on $n$) as in Corollary \ref{definition of the dominated cocycle} satisfying the following:
given arbitrary $q\in \mathbb R^d$ and $\mu \in \M( f)$ there exists $\nu \in \M(T)$ such that
\begin{enumerate}
    \item $nh_{\nu}(T) \leq  h_{\mu}(f) \leq (n+2K_0)h_{\nu}(T)+\frac {n+2K_0}n \log (2K_0+1)$;
    \smallskip
    \item
    $n\lim_{k\to \infty} \frac{1}{k} \int \log  \psi^q(\mathcal{A}^k(\cdot))) \,d\nu 
    \leq 
    \lim_{k \to \infty} \frac{1}{k} \int  \log\psi^q(\mathcal{B}^{k}(\cdot))) \,d\mu$;
    \smallskip
    \item
    $\lim_{k \to \infty} \frac{1}{k} \int  \log\psi^q(\mathcal{B}^{k}(\cdot))) \,d\mu 
    \leq 
    (n+2K_{0})\lim_{k\to \infty} \frac{1}{k} \int \log  \psi^q(\mathcal{A}^k(\cdot))) d\nu$.
\end{enumerate} 
\end{prop}
\begin{proof}
    It follows from the proof of \cite[Proposition 3.4]{Moh23}.
\end{proof}

  \section{Proof of  Theorem \ref{thmA}}\label{Sec: proof of theoremA}

  Let $\A: \Sigma \to \glr$ be a typical cocycle. By Corollary \eqref{definition of the dominated cocycle}, to each large $n\ge 1$ one can associate an alphabet $\mathbb A=\mathbb A_n$ and a dominated cocycle $\B_{\mathbb A}: \mathbb{A}^{\Z} \to \glr$ over the full shift $(\mathbb{A}^{\Z}, f_{\mathbb A}).$ 
  Proposition \ref{prop:add} 
  ensures that 
  the family of potentials $\{\log \sigma_i(\B^n_{\mathbb A})\}_{n\in \N}$  is almost additive. We need the following notations for the proof. We denote by $\mathring{\vec{L}}_{\mathbb{A}}$ the Lyapunov spectrum corresponding to the cocycle $\mathcal{B}_{\mathbb{A}}$, and set
\[
\Omega := \{ (\chi_1(\mu, \mathcal{B}_{\mathbb{A}}), \ldots, \chi_d(\mu, \mathcal{B}_{\mathbb{A}})) : \mu \in \M(f_{\mathbb{A}}) \}.
\]
We also recall that $L$ denotes the Lyapunov spectrum (see \eqref{def:Lyapunov-spectrum}).

  By  \cite[Theorem 5.2 item (1)]{FH}, $\mathring{\vec{L}}_{\mathbb A} \subset \Omega$.
  Thus, since such cocycle is one-step and dominated, combining \cite[Theorem 5.4]{Moh22-entropy} 
  and 
  \cite[Theorem 5.2]{FH} 
  one concludes that 
  \medskip
  \[ 
  \begin{aligned}
h_{\mathrm{top}}(f_{\mathbb A}, E^{n, \mathcal{D}}(\vec{\alpha}))
    & = \inf _{q \in \mathbb{R}^{d}}\left\{P\left(\log \Psi^{q}(\mathcal B_{\mathbb A}) \right)- \langle q, \vec{\alpha} \rangle \right\} \\
    & = \sup\big\{ h_{\mu}(f_{\mathbb A}) \colon   \mu \in \mathcal M_{inv}(f_{\mathbb A}), \; 
    \chi_i(\mu, \mathcal B_{\mathbb A})=\alpha_i, \forall 1\le i \le d\big\}
\end{aligned} 
\]
for every $\alpha \in \mathring{\vec{L}}_{\mathbb A}$, where
\[
E^{n, \mathcal{D}}(\vec{\alpha}):=\bigg\{ x\in \mathbb A^{\Z}: \lim_{n\to \infty}\frac{1}{n}\log \sigma_{i}(\mathcal{B}_{\mathbb A}^{n}(x))=\alpha_i \text{ for } 1\le i \le d \bigg\}.
\]

\smallskip
The idea in the proof of Theorem \ref{thmA} is to estimate the topological entropy of level sets for the original typical cocycle by the one for one-step dominated systems, using the known variational principle for induced subsystems.
Fix $\alpha \in \mathring{\vec{L}}$.
On the one hand, using 
Proposition \ref{relation between entropies and LE} 
and the previous variational principles,
\[ 
\begin{aligned}
   h_{\text{top}} (f_{\mathbb A}, E^{n, \mathcal{D}}( n \vec{\alpha})) 
   & = \sup\big\{ h_{\mu}(f_{\mathbb A}) \colon   \mu \in \mathcal M_{inv}(f_{\mathbb A}), \; 
 \chi_i  (\mu, \mathcal B_{\mathbb A})=  n \alpha_i,  \forall 1\le i \le d\big\} \\
    & \le  \sup\Big\{ (n+2K_0)\, h_{\nu}(T) + \frac {n+2K_0}n \log (2K_0+1) \colon \\
    & \qquad \qquad    \nu \in \mathcal M_{inv}(T), \; 
 \chi_i(\nu, {\mathcal A}) \color{black} \in \Big[\frac{n}{n+2K_0} \alpha_i, \alpha_i \Big],
     \forall 1\le i \le d\Big\}.
\end{aligned}
\]
On the other hand, using 
Proposition \ref{relation between entropies and LE} once more,
\[ 
\begin{aligned}
   h_{\text{top}} (f_{\mathbb A}, E^{n, \mathcal{D}}(n\vec{\alpha})) 
   & \ge  n \cdot \sup\Big\{ \, h_{\nu}(T)  \colon \nu \in \mathcal M_{inv}(T),  \\
    & \qquad \qquad    \; 
     \chi_i(\nu, {\mathcal A})  \in \Big[\frac{n}{n+2K_0}\alpha_i,\alpha_i\Big], \forall 1\le i \le d\Big\}.
\end{aligned}
\]
Therefore, there exists $C>0$ so that, if $n\ge 1$ is large, 
\begin{equation}\label{ThmA: eq1}
   h_{\text{top}} (f_{\mathbb A}, E^{n, \mathcal{D}}( n \vec{\alpha})) \leq (n+2K_{0}) \;
   h_{\text{top}} \Big(T, \bigcup_{\vec \beta \in \Theta_n}E(\vec \beta)\Big) + C
\end{equation}
where 
 the set $\Theta_n$ is formed by all vectors $\vec \beta$ so that $\beta_i \in \Big[\frac{n}{n+2K_0}\alpha_i ,\alpha_i \Big]$ for each $1\le i \le d$. 
Notice that $|\beta_i-\alpha_i |\le \frac{2K_0 \|\vec\alpha\|}{n+2K_0}$ for each $1\le i\le d$.
Then, using Corollary \ref{proof of the upper bound for vectors close to alpha}, we obtain 
\begin{equation}\label{ThmA: eq2}
    h_{\text{top}} \Big(T, \bigcup_{ \vec \beta \in \Theta_n}E(\vec \beta)\Big)
    \leq 
    P (T,\log \Psi^q(\mathcal{A}))
    -\sum_{i=1}^{d}\left(\alpha_i q_i-\frac{2 |q_i| K_0 \|\vec\alpha\|}{n+2K_0}\right)
\end{equation}
for any $q\in \R^d,$ 
in this way, combining the variational principle for the dominated cocycle $\mathcal B_{\mathbb A}$ with inequa\-lities \eqref{ThmA: eq1} and \eqref{ThmA: eq2},
$$
\begin{aligned}
\inf _{q \in \mathbb{R}^{d}}\left\{P_{n, \mathcal{D}}\left(\langle q , \Phi_{\mathcal{B_{\mathbb A}}}-n \vec{\alpha})  \rangle \right) \right\}
    &= h_{\text{top}} (f_{\mathbb A}, E^{n, \mathcal{D}}( n \vec{\alpha}) )\\
& \leq \frac{n+2K_0}{n}\, n\,  h_{\text{top}} \Big(T, \bigcup_{ \vec \beta \in \Theta_n} E(\vec \beta)\Big) + {C}\\
& \leq \frac{n+2K_0}{n} \Big[n\, P (T,\log \Psi^q(\mathcal{A}))-\sum_{i=1}^{d}\left({\alpha_i} q_i-\frac{2 |q_i| K_0 \|\vec\alpha\|}{n+2K_0}\right)\Big] + C.
\end{aligned}
$$
Dividing all terms by $n$, taking the limit as $n$ tends to infinity, and noticing that 
$$
\lim_{n\to\infty} \frac1n P_{n, \mathcal{D}}\left(\langle q , \Phi_{\mathcal{B_{\mathbb A}}}-n\vec{\alpha} \rangle \right)
= 
P\left(T,\langle q ,
\Phi_{\mathcal{{\mathcal A}}}- \vec{\alpha} \rangle \right)
$$ 
(cf. Theorem \ref{continuity_potential}) we conclude that 
 \[
 \begin{aligned}
  \lim_{n\to\infty}
  h_{\text{top}} \Big(T, \bigcup_{\vec \beta \in \Theta_n} E(\vec \beta)\Big) 
  & =
  \lim_{n\to\infty} \frac1n h_{\text{top}} (f_{\mathbb A}, E^{n, \mathcal{D}}(n\vec{\alpha}))
  \color{black} \\
  & = \inf_{q\in \R^d}
  \Big\{\; P \left( \log \Psi^{q}(\mathcal{A})  \right)- \langle q, \vec{\alpha} \rangle \; \Big\} 
  \\
  &  
  =
  \sup\Big\{ \, h_{\nu}(T)  \colon \nu \in \mathcal M_{inv}(T), \; 
     \chi_i(\nu, {\mathcal A}) =\alpha_i, \forall 1\le i \le d\Big\} \\
    & 
    \le h_{\mathrm{top}}(T,E(\vec{\alpha})). 
 \end{aligned}
  \] 
  As $E(\vec\alpha) \subset \bigcup_{\vec \beta \in \Theta_n} E(\vec \beta)$ for each $n\ge 1$, one concludes that 
  $$
h_{\mathrm{top}}(T,E(\vec{\alpha}))
=
  \lim_{n\to\infty}
  h_{\text{top}} \Big(T, \bigcup_{\vec \beta \in \Theta_n} E(\vec \beta)\Big),
  $$
and the proof of the theorem is now complete.

\section{Variational principle for the pressure} \label{sec: proofthmD}

In this section, we proceed to prove Theorem \ref{thmD}. Let $\A: \Sigma \to \glr$ be a fiber-bunched cocycle over a topologically mixing subshift of finite type $(\Sigma, T)$. Let $\mu \in \M(T)$. By Remark \ref{LE exists}, the limit
\[\lim_{n \to \infty} \frac{1}{n}\int \log \psi^{q}(\A^{n}(x)) d\mu(x)\]
is well-defined. We will prove the two inequalities separately.

\begin{lem}\label{upper-bound of var-pri}
 Let $\A:\Si \to \glr$ be a typical cocycle. Then,
 for any $q \in \R^d$,
\[
\sup\bigg\{ h_{\mu}(T)+\lim_{n\to \infty} \frac{1}{n} \int \log \psi^{q}(\A^n(x)) \,d\mu(x): \mu \in \M( T) \bigg\} \leq P(T,\log \Psi^{q}(\mathcal{A})).
\]
\end{lem}
\begin{proof}
 We recall the following inequality (see e.g. \cite{Bow}): given $c_i \in \mathbb{R}, p_i \geq 0$ and $\sum_{i=1}^m p_i=1$ one has that  
 $\sum_{i=1}^m p_i\left(c_i-\log p_i\right) \leq \log \sum_{i=1}^m e^{c_i}.
 $ 
Assume that $\mu \in \M( T).$ Since $\sum_{I \in \mathcal{L}_n} \mu([I])=1$, the inequality above implies that
\[ \frac{1}{n}\sum_{I \in \mathcal{L}_{n}} \mu([I]) \big(-\log \mu([I])+\log \psi^q(\A(I))) \leq \frac{1}{n} \log \sum_{I \in \mathcal{L}_n}  \psi^q(\A_I),\]
 for any $q \in \R^d.$
By Lemma \ref{topological_pressure} and the fact that the limit that defines  $P(\log \psi^q(\mathcal{A}))$ does exist for any $q\in \R^{d}$, one concludes that
\[ h_{\mu}(T)+\lim_{n \to \infty} \frac1{n} \int {\log \psi^q(\A^{n}(x))} \, d\mu(x) \leq P(\log \psi^q(\A)).\]
Since $\mu$ was chosen arbitrary, we obtain the conclusion of the lemma.
\end{proof}

In order to complete the proof of Theorem \ref{thmD} we are left to prove the converse inequality
\[
\sup\bigg\{ h_{\mu}(T)+\lim_{n\to \infty} \frac{1}{n} \int \log \psi^{q}(\A^n(x)) \,d\mu(x): \mu \in \M( T) \bigg\} \geq P(T,\log \Psi^{q}(\mathcal{A})).
\]
We will use the existence of dominated cocycles.
For each $n\ge 1$ let $\mathbb A=\mathbb A_n$ be the
finite alphabet and let 
$\mathcal{B}_{\mathbb A}: \mathbb{A}^{\Z} \to \glr$ be a dominated cocycle  as in Corollary \ref{definition of the dominated cocycle}.  
    Since $\{\langle q, \Phi_{\mathcal{B}_{\mathbb A}} \rangle \}_{n \in \N}$ is almost additive for each $q \in \R^d$ (recall Proposition \ref{prop:add}) and the measure theoretical entropy function is upper semi-continuous we one has the variational principle 
\begin{equation*}\label{var-prin-for-almost addititve}
P_{n, \mathcal{D}}(\log\Psi^q(\mathcal{B}_{\mathbb A}))=\sup\bigg\{ h_{\mu}(f)+\lim_{k \to \infty} \frac{1}{k} \int  \log \psi^{q}(\B_{\mathbb A}^{k}(x))  d\mu(x): \mu \in \M(f)\bigg\},
\end{equation*}
and, for each $q \in \mathbb{R}^d$, the supremum is attained by a unique ergodic equilibrium state $\mu_{n}^q \in \M( f)$  which has the Gibbs property
(cf. \cite[Theorem 10.1.9]{Barreira} for more details). Therefore,
\begin{equation}\label{var-prin-for-almost addititve1}
P_{n, \mathcal{D}}(\log \Psi^q(\mathcal{B}_{\mathbb A}))=h_{\mu_{n}^q}\left(f \right)+\lim_{k \to \infty} \frac{1}{k} \int  \log \psi(\B_{\mathbb A}^k(x))  d\mu_{n}^q(x)
\end{equation}
By Proposition \ref{relation between entropies and LE}, to each $\mu_{n}^q \in \M( f)$ one can associate a $T$-invariant probability measure  $\nu_{n}^q$ on $\Sigma$ such that  \begin{equation}\label{Lyapunov-exponent-relation}
\lim_{k \to \infty} \frac{1}{k} \int \log \psi^{q}(\B^k(x)) d\mu_{n}^q(x) \leq (n+2K_{0})\lim_{k\to \infty} \frac{1}{k} \int \log \psi^{q}(\A^k(x)) d\nu_{n}^q(x) 
\end{equation}
and \begin{equation}\label{relation between entropies}
 h_{\mu_{n}^q}(f) \leq (n+2K_0)h_{\nu_{n}^q}(T)+\frac {n+2K_0}n \log (2K_0+1).
\end{equation} 
In consequence, 
$$
\begin{aligned}
 \frac{1}{n}P_{n, \mathcal{D}}(\log\Psi^q(\mathcal{B}_{\mathbb A})) 
& \stackrel{\eqref{var-prin-for-almost addititve1}}{=}\frac{1}{n} \left( h_{\mu_{n}^q}\left(f \right)+\lim_{k \to \infty} \frac{1}{k} \int \log \psi(\B^k(x)) d\mu_{n}^q(x) \right)\\
&\stackrel{\eqref{relation between entropies}\text{ and } \eqref{Lyapunov-exponent-relation}}{\leq} \frac{n+2K_0}{n} \left(h_{\nu_{n}^{q}}(T) +\lim_{k\to \infty} \frac{1}{k} \int \log \psi^{q}(\A^k(x)) d\nu_{n}^{q}(x) \right) 
\\ & \qquad \quad +\frac {n+2K_0}{n^2} \log (2K_0+1) \\
&
 \leq \frac{n+2K_0}{n} \sup\bigg\{ h_{\nu}(T)+\lim_{k\to \infty} \frac{1}{k} \int \log \psi^{q}(\A^k(x)) d\nu(x): \nu \in \M( T) \bigg\}\\
& \qquad \quad +\frac {n+2K_0}{n^2} \log (2K_0+1).
\end{aligned}
$$
Taking the limit as $n\to+\infty$ and recalling Theorem \ref{continuity_potential} we deduce 
\[
P(T,\log \Psi^{q}(\mathcal{A})) \le \sup\bigg\{ h_{\nu}(T)+\lim_{n\to \infty} \frac{1}{n} \int \log \psi^{q}(\A^n(x)) \,d\nu(x): \nu \in \M( T) \bigg\}.
\]
This finishes the proof of Theorem \ref{thmD}. 


\section{Multifractal formalism of Anosov diffeomorphisms and  repellers}\label{sec:proofapp}

The goal of this section is to describe a multifractal formalism for Anosov diffeomorphisms and repellers whose derivative cocycles fit our assumptions.

\subsection{Proof of Theorem \ref{thmE}}
Assume that $f$ is an Anosov diffeomorphism and that there exists a periodic point $p \in M$ of period $n \in \mathbb{N}$ such that $\left.D_p f^n\right|_{E^u}$ has simple eigenvalues of distinct norms and there exist gives homoclinic points $z_{ \pm} \in M$ of $p$ for $f^n$ whose holonomy loops twist the eigendirections of $\left.D_p f^n\right|_{E^u}$.

By existence of a Markov partition for $f$ \cite{Bow} there exists a H\"older continuous surjection $\pi: \Sigma \rightarrow M$ such that $f \circ \pi=\pi \circ \sigma$, where $\sigma: \Sigma \to \Sigma$ is a topologically mixing subshift of finite type. The assumption on the periodic point $p$ guarantees that, defining  
the cocycle $\A : \Sigma \to GL(d,\mathbb R)$ by \eqref{definition of cocycle for Ansosov},
the cocycle $\A^n: \Sigma \to GL(d,\mathbb R)$ over the shift $(\Sigma,\sigma^n)$ is typical. 
Then the proof follows from Theorem \ref{thmA}.


\subsection{Proof of Theorem \ref{thmF}}

The strategy in the proof of the theorem is to introduce a notion of typicality on repellers (similar to Definition \ref{typical1}), to show that this is satisfied by a large set of repelllers and that it is a sufficient condition to obtain the variational principles.

\begin{defn}
We say that {an} $\alpha$-bunched repeller $\Lambda$ associated to a $C^{1+\alpha}$-map $h$ is \emph{typical} if the following conditions are satisfied:
\begin{itemize}
\item[1)] There exists such a periodic point $p_0 \in \Lambda$ such that the eigenvalues of the matrix $A(p_0):=D_{p_0} h^{\operatorname{per}\left(p_0\right)}$ and all its exterior powers $A(p_0)^{\wedge t}$ $ (1\le t\le d)$ 
have multiplicity $1$ and distinct absolute value;
\item[2)] There exists a sequence of points $\left\{z_n\right\}_{n \in \mathbb{N}_0} \subset \Lambda$ such that
$$
z_0=p_0, \; h (z_n)=z_{n-1} \text {, and } z_n \xrightarrow{n \rightarrow \infty} p_0
$$
and so that, for each $1\le t \le d$, 
the eigenvectors $\left\{v_{1}^{(t)}, \ldots, v_{d_t}^{(t)}\right\}$ of $A(p_0)^{\wedge t}$
are such that, for any $I, J \subset \{1, \ldots, d_t\}$ with $|I|+$ $|J| \leq d_t$, the set of vectors
$$
\left\{ \tilde{H}_{p_0,t}^{\left\{z_n\right\},-} \left(v_{i}^{(t)}\right): i \in I\right\} \cup\left\{v_{j,t}: j \in J\right\}
$$
is linearly independent, where $d_t$ is defined as in ~\eqref{eq:dimwed} and 
\begin{equation}
\label{defholnoninv}
\tilde{H}_{p_0,t}^{\left\{z_n\right\},-}:=\lim _{n \rightarrow \infty}\left((D_{p_0} h)^{\wedge t}\right)^n\left((D_{z_n} h)^{\wedge t} \right)^{-1} \ldots\left((D_{z_1} h)^{\wedge t}\right)^{-1}.
\end{equation}
\end{itemize}
\end{defn}

As in the invertible setting, in case the periodic point $p_0$ given by the previous definition is not fixed, we can consider a power of the $C^{1+\alpha}$-map. In this way we will always assume that $p_0$ is a fixed point of $h$ (see Remark \ref{fixed point}).

\begin{rem}
 It is worth noticing that the choice of a pre-orbit of $p_0$ in item (2) replaces the homoclinic loop in case of invertible maps and that the existence of the limit in ~\eqref{defholnoninv}
is guaranteed by  $C^{\alpha}$-regularity of $D h$  and the $\alpha$-bunching assumption on  $\Lambda$. Indeed, given $p\in \Sigma$ so that $\pi(p)=p_0$, an homoclinic point $z$, the identification $L(\pi p): \mathbb{R}^d \rightarrow T_{p_0} M$ and the cocycle 
$\mathcal{C}$ over $\left(\Sigma, f^{-1}\right)$ 
(cf. \eqref{def:cocliminv} below), the canonical holonomy $H_{ p \leftarrow z}^{s,-}$ is given by
$$
\begin{aligned}
H_{p \leftarrow z}^{s,-}= & \lim _{n \rightarrow \infty} \mathcal{C}^n(p)^{-1} \mathcal{C}^n(z)=\lim_{n \rightarrow \infty} L(\pi p)^{-1} \\
& {\left[\left((D_{p_0} h)^{\wedge t}\right)^n\left((D_{z_n} h)^{\wedge t} \right)^{-1} \ldots\left((D_{z_1} h)^{\wedge t}\right)^{-1}\right] L(\pi z) }
\end{aligned}
$$
and, since $L(\pi z)=L(\pi p)$, $H_{z \leftarrow p}^{u,-}=\text{Id}$. In particular
the holonomy loop for $\mathcal C$, defined by $\tilde{H}_{p, t}^{\{z\},-} :=H_{ p \leftarrow z}^{s,-} \circ H_{z \leftarrow p}^{u,-}$, relates to 
$\tilde{H}_{p_0,t}^{\left\{z_n\right\},-}$
in ~\eqref{defholnoninv} by the conjugacy relation
\begin{equation}\label{relation betwen loops}
   \tilde{H}_{p,t}^{\left\{z\right\},-}=L(\pi p)^{-1} \circ \tilde{H}_{p_0,t}^{\left\{z_n\right\},-} \circ L(\pi p) 
\end{equation}
\end{rem}
\color{black}

\begin{rem} The typicality of $\Lambda$ implies the typicality of the cocycle $\mathcal{C}$ in the sense of Definition~\ref{typical1}. Indeed, $\mathcal{C}(p)$ is equal to the inverse of $D_{p_0} h$ up to the identification $L(\pi p): \mathbb{R}^d \rightarrow T_{p_0} M$. This implies that $\mathcal{C}(p)$ satisfies the pinching condition if $D_{p_0} h$ does. Moreover, since the eigendirections of $\mathcal{C}(p)$ are equal to the eigendirections of its inverse $D_{p_0} h$, \eqref{relation betwen loops} implies that $\tilde{H}_{p,t}^{\left\{z\right\},-}$ satisfies the twisting condition  if $\tilde{H}_{p_0,t}^{\left\{z_n\right\},-}$ does.
\end{rem}

\begin{thm}\label{theorem for the attractor}
 Let $M$ be a Riemannian manifold, and let $h: M \rightarrow M$ be a $C^r$ map with $r>1$. Suppose $\Lambda \subset M$ is an $\alpha$-bunched repeller defined by $h$ for some $\alpha \in(0,1)$ satisfying $r-1>\alpha$.  Then there exist a $C^1$-open neighborhood $\mathcal{V}_1$ of $h$ in $C^r(M, M)$ and a $C^1$-open and $C^r$-dense subset $\mathcal{V}_2$ of $\mathcal{V}_1$ such that $\Lambda_g$ is typical for every $g \in \mathcal{V}_2$ and $$\begin{aligned}
h_{\text{top}}(E(\vec{\alpha}))
&= \inf_{q \in \R^d}\bigg\{P(g,\log \Psi^{q}(\left.D g\right|_{\Lambda_g}) )- \langle \vec{\alpha}, q \rangle \bigg\}\\
&=\sup \bigg\{h_{\mu}(g): \mu \in \M(g), \chi_{i}(\mu, \left.D g\right|_{\Lambda_g})=\alpha_i \text{ for }1\le i \le d \bigg\}.
\end{aligned}
$$
for $\vec{\alpha} \in \mathring{\vec{L}}.$
 \end{thm}

 \begin{proof}
 Let $h\in C^r(M,M)$ be as above.
   By \cite[Lemma 5.10]{park20}, there exists a $C^1$-open neighborhood $\mathcal{V}_1$ of $h$ in $C^r(M, M)$ and a $C^1$-open and $C^r$-dense subset $\mathcal{V}_2$ of $\mathcal{V}_1$ such that $\Lambda_g$ is typical for every $g \in \mathcal{V}_2$. 
   Hence we are left to prove the variational principles.

   \medskip
   Let us first recall some general facts. As $\Lambda$ is a repeller, there exists a finite Markov partition $\mathcal{R}$ for $\Lambda$, a one-sided subshift of finite type $(\Sigma^+,f^+)$, $\Sigma^+\subset \{1,2,\dots, q\}^\N$, and a H\"older continuous and surjective coding map
\begin{equation}
\label{defxi1}
\chi: \Sigma^{+} \rightarrow \Lambda    
\end{equation}
such that $\chi \circ f^{+}=h \circ \chi$. 
By fixing a Markov partition of sufficiently small diameter, we may ensure that the $\chi$-image of each cylinder $[j]$ of $\Sigma^{+}, 1 \leq j \leq q$, is contained in an open set on which $T M$ is trivializable. 

\medskip
Consider the natural extension $\left(\Sigma, f\right)$ of $\left(\Sigma^{+}, f^{+}\right)$, and its inverse $\left(\Sigma, f^{-1}\right.)$.
Let $\pi: \Sigma \rightarrow \Sigma^{+}$ denote the natural projection.
For each $1 \leq j \leq q$ and $y \in[j] \subset \Sigma^{+}$, let $L(y):=L_j(y): \mathbb{R}^d \rightarrow T_{\chi(y)} M$ be a fixed trivialization of $T M$ over an open neighborhood containing $\chi[j]$.  We define a cocycle $\mathcal{C}$ over $\left(\Sigma, f^{-1}\right)$ by
\begin{equation}
\label{def:cocliminv}
\mathcal{C}(x):=L\left(\pi f^{-1} x\right)^{-1} \circ\left(D_{\chi\left(\pi f^{-1} (x)\right)}h\right)^{-1} \circ L(\pi (x)),
\end{equation}
which can be thought of as the inverse of the 
derivative cocycle $D h\mid_\Lambda$
over $\left(\Sigma, f\right)$ defined in the obvious way.
For any $n \in \mathbb{N}$, we have
$$
\mathcal{C}^n\left(f^n (x)\right)=L(\pi x)^{-1}\left(D_{\chi(\pi x)} h^n\right)^{-1} L\left(\pi f^n (x)\right).
$$
Moreover, it relates to $\varphi_{\Lambda, n}^s: \Lambda \rightarrow \mathbb{R}$ by
$$
\varphi^s\left(\mathcal{C}^n\left(f^n (x)\right)\right)=\varphi_{\Lambda, n}^s(\chi(\pi (x))), \quad \forall s\in \mathbb R_{+}.
$$

We proceed to reduce the proof of the theorem to the invertible setting.  
By the proof of \cite[Lemma 5.8]{park20}, for any $\mu \in \M\left(f\right)$ and $\nu \in \mathcal{M}_{\text{inv}}(h)$ related by $\chi_* \mu=\nu$, we have
$$
h_\mu\left(f\right)=h_\nu(h)
\quad \text{and}\quad 
\chi_i(\mu, \mathcal{C})= \chi_i(\nu, Dh)
$$
for each $1\le i \le d$. 
Now, given $g\in \mathcal V_2$ the repeller $\Lambda_g$ is typical and, consequently, there cocycle $\mathcal{C}_g$ defined by \eqref{def:cocliminv} is a typical cocycle.  
Therefore, using the latter and Theorem~\ref{thmD}, we conclude that
$ 
P(g, \log \Psi^{q}(Dg_{|\Lambda_{g}}))=P(f, \log \Psi^{q}(\mathcal{C}_g)).
$  
Then, the result follows from Theorem \ref{thmA}.

 \end{proof}

We are now in a position to finish the proof of Theorem \ref{thmF}. Let $h \in C^r(M, M)$ be as in the statement of the theorem and let $g$ be a $C^1$-small perturbation of $h$. 
If the perturbation is sufficiently small, then 
one can use the same  trivialization over $T_{\Lambda} M$ to code the dynamics of $g$ on $\Lambda_g$ via a conjugacy $\chi_g$ (cf. equation~\eqref{defxi1}), and take its natural extension. Then we realize the perturbation $\left.h\right|_{\Lambda}$ to $\left.g\right|_{\Lambda_g}$ as the perturbation of the cocycle $\mathcal{C}$ to $\mathcal{C}_g$ over the same subshift of finite type $\left(\Sigma, f^{-1}\right.$ ). 
In this way, Theorem~\ref{thmF} is a direct consequence of Theorem \ref{theorem for the attractor}.

\subsection*{Acknowledgements.}
This work was initiated during a research visit of PV to Uppsala University, whose research conditions are greatly acknowledged. 
RM supported by the Knut and Alice Wallenberg Foundation and the Swedish Research Council grant 104651320. PV was partially supported by 
 CIDMA under the
Portuguese Foundation for Science and Technology 
(FCT, https://ror.org/00snfqn58)  
Multi-Annual Financing Program for R\&D Units,
grants UID/4106/2025 and UID/PRR/4106/2025.

\bibliographystyle{abbrv}
\bibliography{MV-finalv2}

@article{Moh23,
title = {Restricted variational principle of {L}yapunov exponents for typical cocycles },
journal = {\url{https://arxiv.org/abs/2301.01721}},
year = {2023},
author = {R. Mohammadpour},
}

@book{horn1994topics,
  author    = {Roger A. Horn and Charles R. Johnson},
  title     = {Topics in Matrix Analysis},
  year      = {1994},
  publisher = {Cambridge University Press}
}

@article{Ruelle1979,
  author    = {David Ruelle},
  title     = {Analyticity properties of the characteristic exponents of random matrix products},
  journal   = {Adv. Math.},
  volume    = {32},
  number    = {1},
  pages     = {68--80},
  year      = {1979}
}

@article{CaVa,
  title={Genericity of historic behavior for maps and flows},
  author={M. Carvalho and P. Varandas},
  journal={Nonlinearity},
  volume={34},
  number={10},
  pages={7030--7044},
  year={2021}
}

@article{HK,
title = {Linear Algebra},
journal = {Springer-Verlag, India, New Delhi},
year = {2000},
author = {K. Hoffman and R. Kunze},
}

@article{Walters,
title = {An Introduction to Ergodic Theory},
journal = {Springer New York, NY},
year = {1975},
author = {P.~Walters}
}

@article{Bochi-Morris,
author = {J. Bochi and I. Morris},
title = {Continuity properties of the lower spectral radius},
journal = {Proc. London Math. Soc.},
volume = {110},
number = {2},
pages = {477-509},
doi = {https://doi.org/10.1112/plms/pdu058},
url = {https://londmathsoc.onlinelibrary.wiley.com/doi/abs/10.1112/plms/pdu058},
eprint = {https://londmathsoc.onlinelibrary.wiley.com/doi/pdf/10.1112/plms/pdu058},
year = {2015}
}

@article{BSS,
  title={Higher-dimensional multifractal analysis},
  author={L. Barreira and B. Saussol and J. Schmeling},
  journal={J. Math. Pures Appl.},
  volume={81},
  number={1},
  pages={67--91},
  year={2002},
}

@article{BaSc,
  title={Sets of ``non-typical" points have full topological entropy and full Hausdorff
dimension},
  author={L. Barreira and J. Schmeling},
  journal={Israel J. Math.},
  volume={116},
  pages={29--70},
  year={2000},
}

@article{IJ15,
  title={Multifractal analysis of {B}irkhoff averages for countable Markov maps},
  author={G. Iommi and T. Jordan},
  journal = {Ergod. Th. Dynam. Sys.},
  volume  = {35},
  pages   = {2559--2586},
  year    = {2015}
}

@article{Olsen,
  title={Multifractal analysis of divergence points of deformed measure theoretical {B}irkhoff
averages},
  author={L. Olsen},
  journal = {J. Math. Pures Appl.},
  volume  = {82},
  number  = {12},
  pages   = {1591--1649},
  year    = {2003}
}

@article{JJOP,
  title={Multifractal analysis of non-uniformly hyperbolic systems},
  author={A. Johansson and T. Jordan and A. \"Oberg and M. Pollicott},
  journal={Israel J. Math.},
  number={177},
  pages={125–144},
  year={2010},
}

@article{BJKR,
  title={Birkhoff and {L}yapunov spectra on planar self-affine sets},
  author={B. Bárány and  T. Jordan and A. Käenmäki and M. Rams},
  journal={Int. Math. Res. Not. IMRN},
  volume={2021},
  number={10},
  pages={7906--8005},
  year={2021},
}

@article{BT22,
  title={Dynamically defined subsets of generic self-affine sets},
  author={B. B{\'a}r{\'a}ny and S. Troscheit},
  journal={Nonlinearity},
    volume={35},
  number={10},
  pages={4986--5013},
  year={2022}
}

@article {Bowen-entropy,
    AUTHOR = {Bowen, Rufus},
     TITLE = {Topological entropy for noncompact sets},
   JOURNAL = {Trans. Amer. Math. Soc.},
    VOLUME = {184},
      YEAR = {1973},
     PAGES = {125--136},
      ISSN = {0002-9947,1088-6850},
   MRCLASS = {28A65 (54H20)},
  MRNUMBER = {338317},
MRREVIEWER = {Karl\ Sigmund},
       DOI = {10.2307/1996403},
       URL = {https://doi.org/10.2307/1996403},
}

@article{BV19,
  title={The gluing orbit property, uniform hyperbolicity and large deviation principles for semiflows},
  author={T. Bomfim and P. Varandas},
  journal={J. Diff. Eq.},
    volume={267},
  pages={228--266},
  year={2019}
}

@article{Katok,
  title={Lyapunov exponents, entropy and periodic orbits for diffeomorphisms},
  author={A. Katok},
  journal={nst.
Hautes Études Sci. Publ. Math.},
    volume={51},
  pages={137--173},
  year={1980}
}

@article{BoVa,
  title={Multifractal analysis of the irregular set for almost-additive sequences via large deviations},
  author={T. Bomfim and P. Varandas},
  journal={Nonlinearity},
    volume={28},
  pages={3563--3585},
  year={2015}
}

@article{BoVa2,
  title={Multifractal analysis for weak {G}ibbs measures: from large deviations to irregular sets},
  author={T. Bomfim and P. Varandas},
  journal={Ergod. Th. Dynam. Sys.},
    volume={37},
  pages={79--102},
  year={2017}
}

@article{Shi,
  title={On multifractal analysis and large deviations of singular-hyperbolic attractors},
  author={Y. Shi and X. Tian and  P. Varandas and X. Wang},
  journal={Nonlinearity},
    volume={36},
  pages={5216},
  year={2023}
}

@article{KS,
  title={Cocycles with one exponent over partially hyperbolic systems},
  author={B. Kalinin and V. Sadovskaya},
  journal={Geom. Dedicata},
    volume={167},
    number = {1},
  pages={167--188},
  year={2013}
}

@incollection{HP,
title = {Chapter 1 - Partially Hyperbolic Dynamical Systems},
editor = {B. Hasselblatt and A. Katok},
series = {Handbook of Dynamical Systems},
publisher = {Elsevier Science},
volume = {1},
pages = {1-55},
year = {2006},
booktitle = {Handbook of Dynamical Systems},
issn = {1874-575X},
doi = {https://doi.org/10.1016/S1874-575X(06)80026-3},
url = {https://www.sciencedirect.com/science/article/pii/S1874575X06800263},
author = {B. Hasselblatt and Y. Pesin}
}

@article{FFW,
author = {A. Fan  and D. Feng and J. Wu},
title = {Recurrence, Dimension and Entropy},
journal = {J. London Math. Soc.},
volume = {64},
number = {1},
pages = {229-244},
doi = {https://doi.org/10.1017/S0024610701002137},
url = {https://londmathsoc.onlinelibrary.wiley.com/doi/abs/10.1017/S0024610701002137},
eprint = {https://londmathsoc.onlinelibrary.wiley.com/doi/pdf/10.1017/S0024610701002137},
year = {2001}
}

@article{MP-uniform-qm,
  title={Uniform quasi-multiplicativity of locally constant cocycles and applications},
  author={R. Mohammadpour and K. Park},
      volume={275},
    number = {1},
  pages={85--98},
  journal={Studia Math.},
    year={2024}
}

@article{O,
  title={Multifractal analysis in symbolic dynamics and distribution of pointwise dimension
for g-measures},
  author={E. Olivier},
  journal={Nonlinearity},
   volume={12},
  pages={1571--1585},
  year={1999}
}

@article{BS21,
  title={The joint spectrum},
  author={E. Breuillard and C. Sert},
  journal={J. London Math. Soc.},
  volume={103},
  number={2},
  pages={943--990},
  year={2021},
}

@article{BochiViana2005,
  author    = {J. Bochi and M. Viana},
  title     = {The {L}yapunov exponents of generic volume-preserving and symplectic maps},
  journal   = {Ann. of Math.},
  volume    = {161},
  number    = {3},
  pages     = {1423--1485},
  year      = {2005}
}

@article{AV07-acta,
  title={Simplicity of {L}yapunov spectra:
proof of the {Z}orich–{K}ontsevich conjecture},
  author={ A. Avila and M. Viana},
  journal={Acta Math.},
   volume={189},
  pages={1--56},
  year={2007}
}

@article{AV10,
  title={Extremal {L}yapunov exponents: an invariance principle and applications},
  author={ A. Avila and M. Viana},
  journal={Invent. Math.},
   volume={181},
  pages={115--178},
  year={2010}
}

@article{Mohammadpour-survey,
    AUTHOR = {Mohammadpour, R.},
     TITLE = {On the multifractal formalism of {L}yapunov exponents: a
              survey of recent results},
journal = {New trends in {L}yapunov exponents, CIM Series in Mathematical Sciences. Springer, Cham},
    SERIES = {CIM Ser. Math. Sci.},
     PAGES = {119--139},
 PUBLISHER = {Springer, Cham},
      YEAR = {2023},
      ISBN = {978-3-031-41315-5; 978-3-031-41316-2},
       DOI = {10.1007/978-3-031-41316-2\_6},
       URL = {https://doi.org/10.1007/978-3-031-41316-2_6},
}

@article{climenhaga, title={The thermodynamic approach to multifractal analysis}, volume={34}, DOI={10.1017/etds.2014.12}, number={5}, journal={Ergod. Theory Dyn. Syst.}, publisher={Cambridge University Press}, author={Climenhaga, V.}, year={2014}, pages={1409–1450}}

@article{BG06,
  title={ Multifractal analysis for {L}yapunov exponents on nonconformal
repellers},
  author={L. Barreira and K. Gelfert},
  journal={Comm. Math. Phys.},
  volume={267},
  number={2},
  pages={393--418},
  year={2006},
}

@article{CPZ,
  title={Dimension Estimates for Non-conformal Repellers and Continuity of Sub-additive Topological Pressure},
  author={Y. Cao and Y. Pesin and Y. Zhao},
  journal={Geom. Funct. Anal.},
    volume={29},
      number={5},
  pages={1325--1368},
  year={2019},
}

@article {Cuneo,
    AUTHOR = {Cuneo, N},
     TITLE = {Additive, almost additive and asymptotically additive
              potential sequences are equivalent},
   JOURNAL = {Comm. Math. Phys.},
    VOLUME = {377},
      YEAR = {2020},
    NUMBER = {3},
     PAGES = {2579--2595},
MRREVIEWER = {Yiwei\ Zhang},
       DOI = {10.1007/s00220-020-03780-7},
       URL = {https://doi.org/10.1007/s00220-020-03780-7},
}

@article{Moh22-entropy,
  author    = {Reza Mohammadpour},
  title     = {Entropy Spectrum of {L}yapunov Exponents for Typical Cocycles},
  journal   = {J. Dyn. Differ. Equ.},
  year      = {2025},
  volume    = {37},
  pages     = {2203--2225},
  doi       = {10.1007/s10884-024-10379-2},
  url       = {https://doi.org/10.1007/s10884-024-10379-2}
}

@article{feng-sh,
  title={Non-conformal Repellers and the Continuity of Pressure for Matrix Cocycles},
  author={ DJ. Feng and P. Shmerkin},
  journal={Geom. Funct. Anal.},
   volume={24},
  pages={1101--1128},
  year={2014}
}

@article{park20,
  title={Quasi-multiplicativity of typical cocycles},
  author={K. Park},
  journal={Comm. Math. Phys.},
  volume={376},
  number={3},
  pages={1957--2004},
  year={2020},
  publisher={Springer}
}

@article{Par22,
  title={Construction and applications of proximal maps for typical cocycles},
  author={K. Park},
  journal={Ergod. Th. Dynam. Sys.},
  pages={1--32},
  year={2022},
  publisher={Springer}
}

@article{Barreira,
  title={\normalfont Thermodynamic formalism and applications to dimension theory},
  author={L. Barreira},
  journal={Progress in
{M}athematics , Birkh{\"a}user},
    volume={294},
  year={2011},

  
}

@book {pesin,
    AUTHOR = {Pesin, Yakov B.},
     TITLE = {Dimension theory in dynamical systems},
    SERIES = {Chicago Lectures in Mathematics},
      NOTE = {Contemporary views and applications},
 PUBLISHER = {University of Chicago Press, Chicago, IL},
      YEAR = {1997},
     PAGES = {xii+304},
      ISBN = {0-226-66221-7; 0-226-66222-5},
   MRCLASS = {58-02 (28A78 54F45 54H20 58F11 58F15 58F99)},
  MRNUMBER = {1489237},
MRREVIEWER = {Irene\ Hueter},
       DOI = {10.7208/chicago/9780226662237.001.0001},
       URL = {https://doi.org/10.7208/chicago/9780226662237.001.0001},
}

@article{Moh22,
  title={Lyapunov Spectrum Properties and Continuity of the Lower Joint Spectral Radius},
  author={R. Mohammadpour},
  journal={J. Stat. Phys.},
  volume={187},
  number={3},
  pages={23},
  year={2022},
  publisher={Springer}
}

@article{falconer94, title={Bounded distortion and dimension for non-conformal repellers}, volume={115}, DOI={10.1017/S030500410007211X}, number={2}, journal={Mathematical Proceedings of the Cambridge Philosophical Society}, author={Falconer, K. J.}, year={1994}, pages={315–334}}

@article{BV04,
  title     = {Lyapunov exponents with multiplicity one for deterministic products of matrices},
  author    = {Bonatti, C. and Viana, M.},
  journal   = {Ergod. Th. Dynam. Sys.},
  volume    = {24},
  number    = {5},
  pages     = {1295--1330},
  year      = {2004},
  publisher = {Cambridge University Press}
}

@article{falconer94b, title={The multifractal spectrum of statistically self-similar measures}, volume={7}, number={3}, journal={J. Theoret. Probab.}, author={Falconer, K. J.}, year={1994}, pages={681--702}}

@article{Bow,
  title={{\normalfont Equilibrium States and the Ergodic Theory of Anosov Diffeomorphisms. {L}ecture
{N}otes in {M}athematics}},
  author={R. Bowen},
  volume={470},
  year={1975},
    journal={Springer, Berlin}
}

@article{Cao,
  title={Dimension spectrum of asymptotically additive potentials for ${C}^{1}$ average conformal repellers},
  author={Y. Cao},
  journal={Nonlinearity},
    volume={26},
  pages={2441--2468},
  year={2013},
}

@article{CFH,
  title   = {The thermodynamic formalism for sub-additive potentials},
  author  = {Y. Cao and D. Feng and W. Haung},
  journal = {Discrete Contin. Dyn. Syst.},
  volume  = {20},
  number  = {3},
  pages   = {639-657},
  year    = {2008}
}

@article{JT21,
  title={Mixed multifractal spectra of {B}irkhoff averages for
non-uniformly expanding one-dimensional Markov
maps with countably many branches
},
  author={J. Jaerisch and H. Takahasi},
  journal={Adv. Math.},
    volume={385},
  pages={107778},
  year={2021},
}

@article{FH,
  title={Lyapunov spectrum of asymptotically sub-additive potentials},
  author={D. Feng and W. Huang},
  journal={Comm. Math. Phys.},
    volume={297},
      number={1},
  pages={1--43},
  year={2010},
}

@article{BGO,
  title={Some characterizations of domination},
  author={J. Bochi and  N. Gourmelon},
  journal={Math. Z.},
    volume={263},
     number={1},
  pages={221--231},
  year={2009},
}
\end{document}